\documentclass[12pt]{amsart} 
\usepackage{amsmath,amssymb,amscd,amsthm}
\usepackage{latexsym}
\usepackage{graphicx}
\usepackage[english]{babel}
\usepackage[latin1]{inputenc}       
\usepackage{textcomp}
\usepackage{times}
\usepackage{pgf,pgfarrows,pgfnodes,pgfautomata,pgfheaps}
\usepackage{colortbl}

\usepackage[a4paper,twoside,left=3cm,right=2.8cm,top=3.1cm,bottom=2.3cm]{geometry}

\newtheorem{thm}{Theorem}
\newtheorem{cor}{Corollary}
\newtheorem{prop}{Proposition}
\newtheorem{theorem}{Theorem}[section]
\newtheorem{corollary}[theorem]{Corollary}
\newtheorem{proposition}[theorem]{Proposition}
\newtheorem{lemma}[theorem]{Lemma}

\newtheorem{definition}[theorem]{Definition}

\newtheorem{remark}[theorem]{Remark}
\newtheorem{example}[theorem]{Example}

\def\irr#1{{\rm Irr}(#1)}
\def\irrr#1#2 {\irr {#1 \mid #2}}
\newcommand{\R}{\mathbb R}
\newcommand{\N}{\mathbb N}

\newcommand{\B}{\mathcal{B}}
\newcommand{\Z}{\mathbb Z}
\newcommand{\sfe}{{{\mathbb S}^{n-1}}}

\newcommand{\E}{\mathbb E}

\begin{document}

\title[Random rounding and singular values]{The smallest singular value of heavy-tailed not necessarily i.i.d. random matrices via random rounding}
\author[Galyna V. Livshyts]{Galyna V. Livshyts}

\address{School of Mathematics, Georgia Institute of Technology} \email{glivshyts6@math.gatech.edu}

\begin{abstract} We show the existence of a net near the sphere, such that the values of any matrix on the sphere and on the net are compared via \emph{regularized Hilbert-Schmidt norm}, which we introduce. This allows to construct an efficient net  which controls the length of $Ax$ for any random matrix $A$ with independent columns (no other assumptions are required).

As a consequence we show that the smallest singular value $\sigma_n(A)$ of an $N\times n$ random matrix $A$ with i.i.d. mean zero variance one entries enjoys the following small ball estimate, for any $\epsilon>0$:
$$P\left(\sigma_n(A)<\epsilon(\sqrt{N+1}-\sqrt{n})\right)\leq (C\epsilon\log{1}/{\epsilon})^{N-n+1}+e^{-cN}.$$
The proof of this result requires working with matrices whose rows are not independent, and, therefore, the fact that the theorem about discretization works for matrices with dependent rows, is crucial.

Furthermore, in the case of the square $n\times n$ matrix $A$ with independent entries having concentration function separated from 1, i.i.d. rows, and such that $\E||A||_{HS}^2\leq cn^2$, one has
$$P\left(\sigma_n(A)<\frac{\epsilon}{\sqrt{n}}\right)\leq C\epsilon+e^{-cn},$$
for any $\epsilon>0$. In addition, for $\epsilon>\frac{c}{\sqrt{n}}$ the assumption of i.i.d. rows is not required. Our estimates generalize the previous results of Rudelson and Vershynin \cite{RudVer-square}, \cite{RudVer-general}, which required the sub-gaussian mean zero variance one assumptions, as well as the work of Rebrova and Tikhomirov \cite{RebTikh}, where mean zero variance 1 and i.i.d. entries were required. 

\end{abstract}
\maketitle

\section{Introduction}


Given a random matrix $A$ (i.e., a matrix with random entries), the question of fundamental interest is: \emph{how likely is $A$ to be invertible}? In other words, how likely is $A$ to not compress the space to a subspace? By linearity, the whole information about $A$ can be obtained from its action on the unit sphere. A matrix which is ``well invertible'' sends the unit sphere to a ``fat'' ellipsoid. 

Suppose $A$ is an $N\times n$ matrix acting from $\R^n$ to $\R^N$. We recall that the axes of the ellipsoid $AB_2^n$ are called singular values of $A$. In other words, singular values 
$$\sigma_1(A)\geq...\geq \sigma_n(A)$$ 
are square roots of the eigenvalues of $AA^T.$ In particular, the largest singular value of $A$ (or the operator norm of $A$) is 
$$\sigma_1(A)=\max_{x\in\sfe} |Ax|,$$
and the smallest singular value
$$\sigma_n(A)=\min_{x\in\sfe} |Ax|;$$
here $|\cdot|$ stands for the euclidean norm. In the case of random matrices with i.i.d. standard Gaussian entries, the largest singular value was known, from the work of Gordon \cite{gordon}, to be of order $\sqrt{N}+\sqrt{n}$, and the smallest singular value to be at least of order $\sqrt{N}-\sqrt{n}$. In addition, the limiting case of $N=n$ has the asymptotic $\frac{1}{\sqrt{n}}$, as was shown by Edelman \cite{edelman} and Szarek \cite{szarek}. As for the operator norm (or the largest singular value), it was shown by Bai and Yin \cite{BaiYin} to be of order $\sqrt{N}$ for random matrices with i.i.d. entries having bounded fourth moment. Furthermore, the assumption of the bounded fourth moment was proved to be necessary by Silverstein \cite{Silverstein}. See also Litvak, Spector \cite{LitSpec} for more details and examples. 

An expression for the largest singular value was obtained by Lata\l{}a \cite{Latala} in the case of independent entries \emph{with different} variances, under the assumption of bounded fourth moments. Recently, a very sharp upper estimate in the case of independent entries with different variances was obtained by Van Handel, Lata\l{}a and Youssef \cite{vHLY}, under the assumption that the entries are independent Gaussian, or, more generally, light-tailed.

Historically, studying the small ball behavior of the smallest singular value has been a more difficult question. Litvak, Rudelson, Pajor, Tomczak-Jaegermann \cite{LPRT} have shown that the smallest singular value is bounded from below by $C\sqrt{N}$ with exponentially high probability, in the case of ``tall'' matrices with sub-gaussian entries, i.e. for those matrices whose aspect ratio $\frac{N}{n}$ is sufficiently large. In a breakthrough paper, Rudelson \cite{Rud-square} has obtained an estimate $C\epsilon n^{-\frac{3}{2}}$ for square (i.e., $n\times n$) matrices whose entries are sub-gaussian i.i.d, with probability $\epsilon,$ provided that $\epsilon>\frac{C}{\sqrt{n}}$. It may seem impossible to obtain a small ball probability estimate for $\inf_{x\in\sfe}|Ax|$ for small $\epsilon>0$, without assuming the uniform small ball property for $|Ax|$ with each fixed $x\in\sfe.$ However, incredibly, Rudelson and Vershynin \cite{RudVer-square} have shown that the smallest singular value of a square random matrix with i.i.d. sub-gaussian entries exceeds $\frac{c\epsilon}{\sqrt{n}}$, with probability $\epsilon+e^{-cn}$, for arbitrary $\epsilon>0,$ thereby confirming the conjecture of von Neumann and Goldstine \cite{Neuman}. Their idea was based on exploring the direction and the arithmetic structure of normals to random subspaces, and a precise analysis yielding the required small ball estimates. This followed up an earlier work of Tao and Vu \cite{TaoVu-LitOf}, \cite{taovu}, where certain related ideas were explored. Further, Rudelson and Vershynin \cite{RudVer-general} have shown, for matrices with arbitrary aspect ratio, under the assumption of i.i.d. entries with mean zero, unit variances and sub-gaussian tails, that
$$P\left(\sigma_n(A)<\epsilon(\sqrt{N+1}-\sqrt{n})\right)\leq \left(C\epsilon\right)^{N+1-n}+e^{-cN}.$$
This behavior was new even in the case of Gaussian matrices, in the case when $N\leq n+O(\sqrt{n})$ (in the complementary case, the small ball bound followed from the work of Gordon \cite{gordon}.)

Further, estimates on the smallest singular value for matrices with arbitrary aspect ratio were done by Tao and Vu \cite{taovu-tall}, Feldheim, Sodin \cite{FS}, and Vershynin \cite{Versh-tall}. 

Note that the exponential additions to the small ball estimates cited above are necessary: for matrices with i.i.d. Bernoulli entries, the smallest singular value is zero if at least two columns or rows coincide, that is at least with probability $2^{-n}$. It was shown by Kahn, Kolmos, Szemeredi \cite{KKS}, and improved by Tao, Vu \cite{taovu-1}, \cite{taovu}, and Bourgain, Vu, Wood \cite{BVW}, that the Bernoulli matrix is invertible with probability $1-e^{-cn}$. The exact value of the constant is $\log 2+o(1),$ as was shown in a breakthrough work by Tikhomirov \cite{TikhErd} very recently.


For square matrices, in addition to the sub-gaussian estimate, Rudelson and Vershynin showed in \cite{RudVer-square} a weak small-ball estimate under only the assumption of the bounded fourth moment. The reason for the fourth moment assumption playing a role for the estimates on the smallest singular value is simply that the ``folklore'' methods use the estimates on the operator norm (since it is the ``Lipschitz constant'' for $|Ax|$ and hence is required in the net argument), which, as we already pointed out, requires the bounded fourth moment. 

For that reason, it remained unclear for a while, if the fourth moment assumption is essential for estimating the smallest singular value. In a breakthrough work, Tikhomirov \cite{Tikh} obtained the limiting behavior for the smallest singular value for the matrices with independent unit variance entries, without assuming bounded fourth moments, in the case of tall matrices. Furthermore, in \cite{Tikh-nomoments} Tikhomirov found an estimate for the smallest singular value of a ``tall'' random matrix, that is, when $N\geq (1+\mu)n$, without assuming any moment constrains. In addition, in the case of square matrices, Rebrova and Tikhomirov \cite{RebTikh} recovered \emph{the full strength} of the sub-gaussian result of Rudelson and Vershynin from \cite{RudVer-square}, assuming only i.i.d mean zero variance one entries. Very recently, Guedon, Litvak, Tatarko \cite{GLT} extended the method of Rebrova and Tikhomirov to get the bound on the smallest singular values of ``tall'' matrices, as well as to study the geometry of random polytopes. In all the cases, the crucial point is to bypass the estimate on the operator norm, by discretizing the sphere in a non-standard way.

Recently, Tikhomirov \cite{Tikh1} found a sharp small ball estimate for square random matrices whose entries have bounded density, which does not depend on moments. Further, in regards to matrices whose entries are not i.i.d., Cook \cite{cook} obtained a general estimate for ``structured'' random matrices. Convergence results for the smallest singular values were obtained under weak assumptions by Bai, Yin \cite{BaiYin-small}, Mendelson, Paouris \cite{MenPao}, and later by Koltchinskii, Mendelson \cite{MenKol}, and others. We omit a detailed discussion about the long and rich history of the smallest singular values estimates, and refer the reader to a survey by Rudelson, Vershynin \cite{RudVer-icm}. 

\medskip
\medskip

We shall need to assume the following property

\begin{definition}\label{concentration}
A random variable $X$ is said to have concentration function separated from 1 (or, for brevity, bounded concentration function), if there exist absolute constants $a\in\R^+$ and $b\in (0,1)$ such that
$$\sup_{z\in\R}P(|X-z|<a)<b.$$ 
\end{definition}

\begin{remark}
Note that a priori, without assuming that the entries have bounded concentration function, we cannot have a big smallest singular value: for instance, what if all the entries take one value with very large probability? Then we are very likely to have a pair of equal rows or columns, in which case the smallest singular value would be 0.
\end{remark}

We begin to formulate our results.

\begin{thm}\label{main1}
Fix any $p>0$. Let $A$ be an $n\times n$ random matrix with independent entries, and suppose that the entries of $A$ have uniformly bounded concentration function. Assume that there exists an absolute constant $K>0$ such that 
\begin{equation}\label{HSbound-11}
\sum_{i=1}^n \left(\E|Ae_i|^{2p}\right)^{\frac{1}{p}}\leq Kn^2,
\end{equation}
and 
\begin{equation}\label{HSbound-12}
\sum_{i=1}^n \left(\E|A^T e_i|^{2p}\right)^{\frac{1}{p}}\leq Kn^2.
\end{equation}
Then: 
\begin{itemize}
\item for every $\epsilon>\frac{c}{\sqrt{n}}$,
\begin{equation}\label{keyoutcome}
P\left(\sigma_n(A)<\frac{\epsilon}{\sqrt{n}}\right)\leq C\epsilon+e^{-c\min(p,1)n},
\end{equation}
where $C$ and $c$ are absolute constants which depend (polynomially) only on $K$ from (\ref{HSbound-1}) and $a$ and $b$ from Definition \ref{concentration}, and $C$ may additionally depend on $p.$

\item If, in addition, the rows of $A$ are identically distributed, we have (\ref{keyoutcome}) for every $\epsilon>0$.
\end{itemize}
\end{thm}

Theorem \ref{main1} generalizes the results of Rudelson, Vershynin \cite{RudVer-square} and Rebrova and Tikhomirov \cite{RebTikh}. We emphasize that in part 1 of the theorem, very mild assumptions are placed on the matrix, and various structural assumptions such as mean zero, equal variance, i.i.d. entries are not required. The assumption of i.i.d. rows in part 2 is required, since it is needed for deep analytic small ball estimates of Rudelson and Verhsynin \cite{RudVer-square}, \cite{RudVer-general}, and removing this assumption appears rather difficult. We emphasize that part 2 of Theorem \ref{main1} is, in fact, a much deeper fact than part 1, although it does require an extra assumption, and the complexity of the proof for part 2 is a lot greater.

For completeness, we outline a corollary in the case $p=1$.

\begin{cor}\label{main1cor}
Let $A$ be an $n\times n$ random matrix with independent entries. Suppose further that the entries of $A$ have bounded concentration function, and that there exists an absolute constant $K>0$ such that 
\begin{equation}\label{HSbound-1}
\E ||A||_{HS}^2\leq Kn^2.
\end{equation}
Then for every $\epsilon>\frac{c}{\sqrt{n}}$,
$$P\left(\sigma_n(A)<\frac{\epsilon}{\sqrt{n}}\right)\leq C\epsilon+e^{-cn},$$
where $C$ and $c$ are absolute constants which depend (polynomially) only on $K$ from (\ref{HSbound-1}) and $a$ and $b$ from Definition \ref{concentration}. Further, if the rows of $A$ are identically distributed, the conclusion follows for every $\epsilon>0.$
\end{cor}

\begin{example}\label{example1}
Consider a random matrix $B$ with i.i.d. variance one entries (for example, Bernoulli $\pm1$), and let $A=[\sqrt{n}Be_1, Be_2,...,Be_n]$. I.e., consider a matrix whose first column is $\sqrt{n}$ times heavier than the rest. By Corollary \ref{main1cor}, the smallest singular value of $A$ has the same small ball behavior as the smallest singular value of $B,$ for $\epsilon>0.$ 

Alternatively, we may let $A_1=(a_{ij})$ with $a_{11}=\pm n$ and $a_{ij}=\pm 1$ for $i,j\neq 1$. Then for every $\epsilon>\frac{100}{\sqrt{n}},$ we have 
$$P\left(\sigma_n(A_1)<\frac{\epsilon}{\sqrt{n}}\right)\leq C\epsilon.$$
It is worth noting that the operator norm of $A_1$ is in fact of order $n$ with high probability.
\end{example}

Next, we outline another corollary of Theorem \ref{main1}, with slight change of notation.

\begin{cor}\label{main2cor}
Fix any $p>0.$ Let $A$ be an $n\times n$ random matrix with i.i.d. entries. Suppose further that the entries of $A$ have bounded concentration function, and that there exists an absolute constant $K>0$ such that 
\begin{equation}\label{HSbound-2}
\left(\E |Ae_1|^{p}\right)^{\frac{1}{p}}\leq K\sqrt{n}.
\end{equation}
Then for every $\epsilon>0$,
$$P\left(\sigma_n(A)<\frac{\epsilon}{\sqrt{n}}\right)\leq C\epsilon+e^{-c\min(1, p)n},$$
where $C$ and $c$ are absolute constants which depend (polynomially) only on $K$ from (\ref{HSbound-2}) and $a$ and $b$ from Definition \ref{concentration}, and $C$ additionally depends on $p$. 
\end{cor}

\begin{example}\label{example2}
Consider an $n\times n$ matrix $A$ with i.i.d. entries $a_{ij}$, each distributed according to the density 
\[
	f(s)= 
	\begin{cases}
	\frac{1}{2\sqrt{s}} ,& s\in[-\frac{1}{n^2}, \frac{1}{n^2}],\\
	\frac{1}{2}\cdot\frac{1-\frac{4}{n}}{1-\frac{1}{n^2}}, & s\in [-1,1]\setminus [-\frac{1}{n^2}, \frac{1}{n^2}],\\
	\frac{1}{2ns^3}, & |s|\in [1,\infty).
	\end{cases}
	\]
	
	Note that $\E a_{ij}=0$, but $\E a_{ij}^2=\infty,$ and therefore the result of Rebrova and Tikhomirov \cite{RebTikh} is not applicable to estimate the smallest singular value of $A$. Further, the density of $a_{ij}$ is unbounded, and hence the result of Tikhomirov \cite{Tikh1} (about random matrices whose entries have bounded density) is not applicable either. However, Corollary \ref{main2cor} asserts that
	$$P\left(\sigma_n(A)<\frac{\epsilon}{\sqrt{n}}\right)\leq C\epsilon+e^{-cn}.$$
Indeed, note that for $t\geq \sqrt{n}$,
$$P\left(\sum_{j=1}^n a_{ij}^2\geq t^2\right)\leq nP\left(|a_{11}|\geq \frac{t}{\sqrt{n}}\right)=2n\int_{\frac{t}{\sqrt{n}}}^{\infty} \frac{1}{2ns^3} ds=\frac{2n}{t^2}.$$
Consequently,
$$\E\sqrt{\sum_{i=1}^n a_{1i}^2}=\int_0^{\infty} P\left(\sum_{j=1}^n a_{ij}^2\geq t^2\right) dt\leq$$
$$\sqrt{n}+\int_{\sqrt{n}}^{\infty}\frac{2n}{t^2} dt=3\sqrt{n},$$
and Corollary \ref{main2cor} is hence applicable with $p=1$ and $K=3.$	
\end{example}

\begin{remark} In fact, in Example \ref{example2} it is crucial that the distribution of $a_{ij}$ has ``minor pathology'', i.e. depends on $n;$ the author is grateful to Sergey Bobkov for relating to her this fact. If the distribution of $a_{ij}$ is fixed, and does not depend on $n$, and $\E a_{ij}^2=\infty$, while $\E |a_{ij}|=1$, then necessarily $\E|Ae_1|>>\sqrt{n}$; this follows from Feller's law of large numbers \cite{feller}, which sharpened Kolmogorov's SLLN. A way to get an intuition about this fact would be to estimate $|Ae_1|\geq \|Ae_1\|_{\infty}$, to use Markov's inequality and reduce the conclusion to an inequality about the distribution of $\max (|a_{11}|,...,|a_{n1}|)$, which boils down to an assertion about the heavy tails of $|a_{11}|$.

Therefore, Corollary \ref{main2cor} is not applicable for a random matrix with $\E|a_{ij}|=1$, unless either $\E a_{ij}^2=C$ (and then the result of Rebrova and Tikhomirov \cite{RebTikh}, or Corollary \ref{main1cor} are applicable), or the distribution of $a_{ij}$ is pathological, and depends on $n.$ Theorem \ref{main1}, however, does show that minor pathologies, such as the one described in Example \ref{example2}, do not affect the small ball behavior of $\sigma_n(A).$

Besides the Example \ref{example2}, there is another, more probabilistic way of constructing such pathological distribution, which was related to the author by Sergey Bobkov. Namely, consider a nice distribution $Y$ (for example, Bernoulli $\pm 1$), such that $\E |Y|^2=1$, and a bad distribution $Z$, such that $\E |Z|=1$ but $\E Z^2=\infty$. Suppose also that $Y$ and $Z$ are independent, and $\E Y=\E Z=0$. Consider
$$X=\sqrt{(1-\frac{1}{n})Y^2+\frac{1}{n}Z^2},$$
and let $A$ be a random matrix with i.i.d. entries $a_{ij}$ distributed as $X$. Denote
$$\tilde{Y}=\sqrt{Y_1^2+...+Y_n^2}$$
and
$$\tilde{Z}=\sqrt{Z_1^2+...+Z_n^2},$$
where $Y_i$ and $Z_i$ are i.i.d. copies of $Y$ and $Z$ correspondingly. Then
$$\E|Ae_1|\leq \sqrt{1-\frac{1}{n}}\E \tilde{Y}+\frac{1}{\sqrt{n}}\E \tilde{Z}\leq C\sqrt{n}.$$

\medskip
\medskip
Let us also note, that in fact it is not very clear if one should in general expect the nice small ball behavior when the square matrix has, say, i.i.d. entries with $\E|a_{ij}|=1$ but $\E a_{ij}^2=\infty.$
\end{remark}


\medskip

Next, we formulate a result concerning matrices with arbitrary aspect ratio.

\begin{thm}\label{main}
Let $N\geq n\geq 1$ be integers. Let $A$ be an $N\times n$ random matrix with mean zero independent entries, and suppose additionally that the rows of $A$ are identically distributed. Suppose further that the entries of $A$ have bounded concentration function, and that there exists an absolute constant $K>0$ such that for every $\sigma\subset\{1,...,n\}$ with $$\#\sigma=\min(N-n+1, n),$$ we have
\begin{equation}\label{HSbound}
\E\frac{1}{\#\sigma} \sum_{i\in\sigma}|Ae_i|^2\leq KN.
\end{equation}
Then for every $\epsilon>0$,
$$P\left(\sigma_n(A)<\epsilon(\sqrt{N+1}-\sqrt{n})\right)\leq (C\epsilon\log{1}/{\epsilon})^{N-n+1}+e^{-cN},$$
where $C$ and $c$ are absolute constants which depend (polynomially) only on $K$ from (\ref{HSbound}) and $a$ and $b$ from Definition \ref{concentration}. 
\end{thm}

We would like to point out, that in the ``tall'' case, when $N\geq C_0 n$ for a sufficiently large absolute constant $C_0>0$, the implications of Theorem \ref{main} follow under significantly weaker assumptions, as shall be seen in Section 5 in Proposition \ref{tall-final}.

\begin{remark}\label{uneavenness}
Note that the bounded concentration function condition imposes that there exists an absolute constant $c_1>0$ such that for all entries we have $\mathbb{E}a_{ij}^2\geq c_1$. Therefore, $\E \sum_{i\in\sigma}|Ae_i|^2\geq c_1\#\sigma\cdot N.$ Together with (\ref{HSbound}) it determines, up to a multiplicative constant, the value of $\E \sum_{i\in\sigma}|Ae_i|^2.$ However, it does not forbid the situation when one the of columns is significantly heavier than the rest, as was outlined in Example 1.
\end{remark}

For the sake of completeness, we state the following corollary of Theorem \ref{main}.

\begin{cor}\label{coriid}
Let $N\geq n\geq 1$ be integers. Let $A$ be an $N\times n$ random matrix with mean zero i.i.d. entries $a_{ij}$ which satisfy $\E a_{ij}^2=1$. Suppose further that $a_{ij}$ have bounded concentration function. Then for every $\epsilon>0$,
$$P\left(\sigma_n(A)<\epsilon(\sqrt{N+1}-\sqrt{n})\right)\leq \left(C\epsilon\log{1}/{\epsilon}\right)^{N-n+1}+e^{-cN},$$
where $C$ and $c$ are absolute constants which depend (polynomially) only on $a$ and $b$ from Definition \ref{concentration}. 
\end{cor}

\medskip

Lastly, for completeness, we outline further applications of our method in the context of works by Mendelson, Paouris \cite{MenPao} and Koltchinskii, Mendelson \cite{MenKol}: we formulate an estimate in the regime of dependent columns of the matrix.

\begin{prop}\label{tallnodep}
Suppose $A$ is an $N\times n$ random matrix with independent rows, and assume that the rows of $A$ satisfy point-wise small ball assumption: for every $x\in\sfe,$
\begin{equation}\label{smb}
\sup_{y\in\R}P(|\langle A^Te_i, x\rangle-y|\leq a)\leq b,
\end{equation}
for some fixed constants $a\in \R$ and $b\in (0,1)$. Suppose further that 
$$\E||A||_{HS}^2\leq KNn,$$
for some $K>0.$ Then there exists $C_0>0$ depending only on $K,a,b$, such that, provided that $N\geq C_0n,$ we have 
$$\E\sigma_n(A)\geq c\sqrt{N},$$
for some $c>0$ depending only on $K,a,b.$
\end{prop}
We note that Proposition \ref{tallnodep} is outlined here since it follows from our method in a straightforward manner; to obtain a more precise and general statement, a lot more work is required, and all such considerations shall be done separately.

\medskip
\medskip

The key notion which we employ, for a matrix $A,$ and a parameter $\kappa>1$ is the \emph{regularized Hilbert-Schmidt norm}
$$\B_{\kappa}(A):=\min_{\alpha_{i}\in [0,1],\,\prod_{i=1}^n \alpha_i\geq \kappa^{-n}} \sum_{i=1}^n \alpha_i^2 |Ae_i|^2.$$

Let us turn to describing how the notion of the regularized Hilbert-Schmidt norm arises, and discussing the strategy of the proofs. Net arguments are often an important part of obtaining various results in asymptotic analysis, including estimates for singular values of random matrices. It is a classical fact that there exists a set $\mathcal{N}\subset \sfe$ of cardinality $\left(\frac{3}{\epsilon}\right)^n$ such that for every $x\in\sfe$ there exists $y\in \mathcal{N}$ such that $|x-y|\leq \epsilon$, and consequently, for any (deterministic) $n\times n$ matrix $A$, we have
\begin{equation}\label{norm-comp}
|A(x-y)|\leq \epsilon\cdot \|A\|.
\end{equation} 
In fact, one may improve (\ref{norm-comp}), as was outlined in the work of Klartag and the author \cite{KLkold} in Lemma 5.1 (see also Lemma \ref{HS} below). Indeed, consider a lattice partitioning of $\R^n$ into cubes of side length $\frac{\epsilon}{\sqrt{n}}$. For every vector $x\in\R^n$, consider a random vector $\eta^x$ which takes values in the vertices of the cube of this partition in which $x$ belongs, in such a way that the coordinates of $\eta^x$ are independent, and $\E(x-\eta^x)=0$. This procedure is called \emph{random rounding}, see Section 3 for more details and history. The aforementioned two properties yield that for any (deterministic) vector $\theta$ in $\R^n$,    
$$\E\langle x-\eta^x,\theta\rangle^2=\sum_{i=1}^n \theta_i^2|x_i-\eta^x_i|^2\leq \frac{\epsilon^2|\theta|^2}{n}.$$
This implies, in fact, that any closed strip of width $2$, centered at any point inside the unit cube, catches at least one of the vertices of the cube. Next, using the above, for any (deterministic) $n\times n$ matrix $A$, we observe, by summing up:
$$\E|A(x-\eta^x)|^2=\E\sum_{j=1}^n \langle A^Te_j,x-\eta^x\rangle^2\leq \frac{\epsilon^2}{n}\|A\|^2_{HS}.$$
For a vector $x$, denote by $S_x$ the collection of values which $\eta^x$ takes with non-zero probability; in the generic situation, $\# S_x=2^n.$ A crucial observation here is that
$$\cup_{x\in\sfe} S_x\subset \mathcal{F},$$
where $\mathcal{F}=\frac{\epsilon}{\sqrt{n}}\Z^n\cap\frac{3}{2}B_2^n\setminus \frac{1}{2}B_2^n$, and one may infer via an elementary argument that $\#\mathcal{F}\leq \left(\frac{10}{\epsilon}\right)^n$. This implies that there exists a net $\mathcal{F}$ not far from the sphere, of cardinality $(10/\epsilon)^n$, such that for every $x\in\sfe$ there exists $y\in\mathcal{F}$ such that for any deterministic $n\times n$ matrix $A$,
\begin{equation}\label{HS-comp-intro}
|A(x-y)|\leq \frac{\epsilon}{\sqrt{n}}\|A\|_{HS}.
\end{equation}
Note that the assertion (\ref{HS-comp-intro}) is strictly stronger that the assertion (\ref{norm-comp}). Let us now recall the example 1.3: consider an $n\times n$ Bernoulli $\pm 1$ matrix $M$, and let $A=[\sqrt{n}Me_1, Me_2,...,Me_n]$. Then $\| A\|_{HS}=\sqrt{2n^2-1}$ for every realization of $A$, while $\|A\|\geq |Ae_1|=n$, for every realization (and, in fact, with exponentially high probability, $\|A\|\leq n+C\sqrt{n}$ by the triangle inequality combined with the standard result about the norm of the Bernoulli matrix). Consequently, for all the realizations of the matrix $A$, the net $\mathcal{F}$ which gives (\ref{HS-comp-intro}) is always $\sqrt{n}$ times more precise than the net $\mathcal{N}$ with the property (\ref{norm-comp}).

Next, let us consider an even more radical example: let $B=[e^{n}Me_1, Me_2,...,Me_n]$. For this matrix, both (\ref{norm-comp}) and (\ref{HS-comp-intro}) yield the comparison
$$|B(x-y)|\leq \epsilon e^n(1+o(1)).$$
However, instead of $\mathcal{F}$, one may consider a net 
$$\mathcal{S}=\left(\frac{e^{-n}\epsilon}{\sqrt{n}}\Z\times\frac{\epsilon}{\sqrt{n}}\Z\times...\times \frac{\epsilon}{\sqrt{n}}\Z\right)\cap\frac{3}{2}B_2^n\setminus \frac{1}{2}B_2^n.$$
The cardinality of $\mathcal{S}$ is bounded by $\left(\frac{5e}{\epsilon}\right)^n$ (which works essentially just as well as $\left(\frac{5}{\epsilon}\right)^n$ for our purposes), but one may notice, using an argument similar to the one outlined above, that for every $x\in\sfe$ there exists $y\in\mathcal{S}$ such that for every realization of $B,$
$$|B(x-y)|\leq \epsilon\sqrt{n},$$
which improves significantly upon the previous comparison. This leads to the notion of the regularized Hilbert-Schmidt norm $\B_{\kappa}.$ Indeed, note that for the matrix $B,$
$$\B_e(B)=n^2<<(1+o(1))e^{2n}=||B||^2_{HS}.$$
We shall consider nets formed by vertices of parallelepipeds of exponentially small volume, which will lead us to a construction of a net, such that for every $x\in\sfe$ one may find $y$ in the net with
\begin{equation}\label{B-kappa-comp-intro}
|A(x-y)|\leq \frac{\epsilon}{\sqrt{n}}\sqrt{\B_{\kappa}(A)},
\end{equation}
for any deterministic matrix $A$. Moreover, our net will not depend on the matrix. While each matrix might have its own preferred parallelepiped generating the net, we will show that only an exponential number of the parallelepiped shapes are required, by discretizing the set of admissible parallelepipeds.

In addition to the above examples, one may consider a random matrix $C$ with $c_{ij}=\pm 1$ for $i+j>2$ and $c_{11}=\pm 1.01^n$. While both operator and Hilbert-Schmidt norms of $C$ are exponentially large with probability 1, we note that $\B_{1.01}(C)\leq n^2,$ and, moreover, $\B_{1.01}(C^T)\leq n^2$, with probability 1. This two inequalities ensure that, in fact, the conclusion of Theorem \ref{main1} holds for the matrix model $C$, even though formally it is not included in the assumptions. In fact, a more general result of this type holds, as will be discussed in Remark \ref{spiky}.

The regularized Hilbert-Schmidt norm $\B_{\kappa}$ is very useful not only for these inhomogeneous highly spiky matrix profiles, but also in the situation of a heavy-tailed matrix with a homogeneous profile. The reason for this is that a sizable portion of the realizations of a heavy-tailed matrix happen to be very spiky. Indeed, consider a random matrix $D$ with independent columns such that $\E d^2_{ij}=1,$ but such that $\E |d|^{3}_{ij}=\infty$, say. Then with considerable probability, $\|D\|\approx n,$ $\|D\|_{HS}>>n,$ however we shall show that $\B_{\kappa}(D)\approx n^2$ with exponentially good probability.  

Another example is Example 1.4: Let $E$ be a random matrix with independent columns such that $\E e^2_{ij}=\infty$, but such that $\E|Ee_i|^p=n^p,$ for a small $p>0.$ Then with high probability, $\|E\|=\|E\|_{HS}=\infty$, however $\B_{\kappa}(E)\approx n^2$, as will be shown in Section 3.

Let us now turn to formulating our key result about net approximation, which will be entailed by the strong large deviation properties of the regularized Hilbert-Schmidt norm.

\begin{thm}[the net bound]\label{keytheoremnets}
Fix $n\in\N.$ Consider any $S\subset\R^n.$ Pick any $\gamma\in(1,\sqrt{n})$, $\epsilon\in (0,\frac{1}{20\gamma})$, $\kappa>1,$ $p>0$ and $s>0.$ Recall that $N(S, \epsilon B_2^n)$ stands for the covering number of $S$ by $\epsilon B_2^n$ (see more details in section 2).
 
There exists a (deterministic) net $\mathcal{N}\subset S+4\epsilon\gamma B_2^n$, with 
\[
	\#\mathcal{N}\leq 
	\begin{cases} N(S, \epsilon B_2^n)\cdot(C_1 \gamma)^{\frac{C_2 n}{\gamma^{0.08}}},& \text{if}\,\,\, \log\kappa\leq \frac{\log 2}{\gamma^{0.09}},\\
	N(S, \epsilon B_2^n)\cdot(C\kappa)^n(\log\kappa)^{n-1},& \text{if}\,\,\, \log\kappa\geq \frac{\log 2}{\gamma^{0.09}},\\
	\end{cases}
	\]
	such that for every $N\in\N$ and every random $N\times n$ matrix $A$ with independent columns, with probability at least 
$$1-\kappa^{-2pn}\left(1+\frac{1}{s^p}\right)^n,$$
for every $x\in S$ there exists $y\in\mathcal{N}$ such that 
\begin{equation}\label{netmain1}
|A(x-y)|\leq C_3\frac{\epsilon\gamma \sqrt{s}}{\sqrt{n}}\sqrt{\sum_{i=1}^n \left(\E|Ae_i|^{2p}\right)^{\frac{1}{p}}}.
\end{equation}
Here $C, C_1, C_2, C_3$ are absolute constants.
\end{thm}

In order to derive Theorem \ref{keytheoremnets}, we use some ideas from Rebrova and Tikhomirov \cite{RebTikh}. Theorem \ref{keytheoremnets} generalizes Theorem A from their paper, in which it was obtained in the particular case of square random matrix $A$ having all entries i.i.d., mean zero and second moment 1. The net from the paper of Rebrova and Tikhomirov, however, has an advantage of being a subset of the set $S,$ rather than its euclidean neighborhood; this does not play any role in the set up of this paper, and we believe that the net in the neighborhood shall be sufficient also for other potential applications.

\begin{remark}
We emphasize that the only assumption in Theorem \ref{keytheoremnets} is the independence of columns of $A$; of course, the statement is only applicable in the case the right hand side of (\ref{netmain1}) is bounded. The possibility to consider matrices with dependent rows will be crucial in Section 8, when the Theorem \ref{keytheoremnets} will be applied to projections of our matrix $A$.
\end{remark}

\begin{remark} It often occurs in asymptotic analysis, that a logarithmic factor in an estimate appears as a byproduct of the proof, and is, in fact, unnecessary. Clearly, this is the case, for example, with Theorem \ref{main}. In contrast to that, we believe that in Theorem \ref{keytheoremnets} the term $\log \kappa$ might in fact be necessary. 
\end{remark}

\begin{remark} We note that for any set $S\subset\sfe$, the theorem entails that $\mathcal{N}\subset \frac{3}{2}B_2^n\setminus \frac{1}{2}B_2^n$, in view of our assumption on $\epsilon.$ This is crucial throughout the paper, and also is sufficient for all the applications of the net, with the exception of its application in Section 6, where the neighborhood plays a role.
\end{remark}

Theorem \ref{keytheoremnets} might appear overloaded with parameters, however we really need them all. Throughout the paper, we apply Theorem \ref{keytheoremnets} in several situations with different choices of parameters and in completely different regimes, and the precise estimates for the cardinality of the net, the probability of the good event, and the net approximation are crucial, although they might appear incomprehensible at the first glance. For the reader's benefit, we formulate below a corollary of Theorem \ref{keytheoremnets}:

\begin{corollary}\label{cor-simple}
There exists an absolute constant $C>0$ and a deterministic net $\mathcal{N}\subset\frac{3}{2}B_2^n\setminus \frac{1}{2}B_2^n$ of cardinality $1000^{n-1}$ such that for any random matrix $A$ with independent columns, with probability at least $1-e^{-5n}$, for every $x\in\sfe$ there exists $y\in\mathcal{N}$ such that 
$$|A(x-y)|\leq \frac{C}{\sqrt{n}}\sqrt{\E||A||_{HS}^2}.$$
\end{corollary}

In fact, the net from Corollary \ref{cor-simple} can be written out explicitly. 

In Section 2 we discuss rather standard covering estimates, which involve some sparsity arguments. In Section 3 we prove Theorem \ref{keytheoremnets}; furthermore, we prove Theorem \ref{main-net} -- a deterministic version of Theorem \ref{keytheoremnets}, -- which could be of independent interest. From that moment on, the proof follows the scheme developed by Rudelson and Vershynin \cite{RudVer-square}, \cite{RudVer-general}. In Section 4 we survey some of the powerful small ball estimates of Rudelson and Vershynin \cite{RudVer-square}, \cite{RudVer-general}, on which our proof heavily relies. In Section 5 we prove Proposition \ref{tallnodep}, as well as the ``tall'' case, and discuss the decomposition of the sphere introduced by Rudelson and Vershynin; in this section we apply our net for the first time. In Section 6 we generalize results of Rudelson and Vershynin about structure of random subspaces to heavy-tailed matrices; in this section we apply our net for the second time. In Section 7 we prove Theorem \ref{main1}. In Section 8 we prove Theorem \ref{main}, and apply our net for the third time (in fact, many times): the argument involves an iteration in the parameter $\kappa$. 

\textbf{Acknowledgements.}  First of all, the author is thankful to Bo'az Klartag for his support, encouragement, enduring patience and inspiration during the time this paper was written; this project sprouted as a follow up to our joint work \cite{KLkold}, however Bo'az decided not to join it at this point. In addition, the author would like to thank Dr. Klartag for mentoring her during the Fall 2017 MSRI program, which has led to a phase transition in her knowledge, understanding and proficiency in the subject.

The author is very grateful to Mark Rudelson for encouragement and helpful discussions, which have led to a significant improvement of her understanding of the field in general. The author is grateful to Konstantin Tikhomirov for helpful discussions, and in particular for relating to her what types of behaviors of the smallest singular values are expected by the experts. The author is grateful to Roman Vershynin for bringing to her attention the question about matrices whose entries have different moments.  Thanks to Edward Zeng for spotting a typo.

The author is also grateful to the anonymous referee for many valuable comments which helped to significantly improve the presentation and motivated further explorations.

The author is supported in part by the NSF CAREER DMS-1753260. The work was partially supported by the National Science Foundation under Grant No. DMS-1440140 while the author was in residence at the Mathematical Sciences Research Institute in Berkeley, California, during the Fall 2017 semester.

\section{Covering estimates}

We begin by outlining some notation which we have to use throughout the paper; whenever the reader stumbles upon an unknown notation, she may consult with the list below; some more notation is added also in the beginning of the other sections.

\medskip

\textbf{Notation.}
\begin{itemize}
\item We work in an $n$-dimensional euclidean space $\R^n.$ Scalar product is denoted $\langle \cdot,\cdot\rangle$. Euclidean norm is denoted by $|\cdot|,$ and the infinity norm by $||x||_{\infty}=\max_i |x_i|$. The unit ball is denoted $B_2^n$ and the unit sphere $\sfe$. The unit cube is
$$B_{\infty}^n=\{x\in\R^n:\,||x||_{\infty}\leq 1\},$$
and the cross-polytope $B_1^n=\{x\in\R^n:\,\sum_{i=1}^n |x_i|\leq 1\}.$

\item Minkowski sum of sets $A, B\subset \R^n$ is $A+B=\{x+y:\,x\in A,\, y\in B\}.$

\item The integer part of a real number $a$ (i.e., the largest integer which is smaller than $a$) is denoted by $[a]$.
 
\item We say that a set $S\subset\R^n$ can be covered by $m$ translates of a set $K\subset\R^n$ if there exists a collection $x_1,...,x_m\in\R^n$ such that
$$S\subset \cup_{i=1}^m (K+x_i).$$
We use the standard notation $N(S,K)$ for the minimal number of translates of $K$ needed to cover $S$. 

\item Given a parallelepiped $P$, we also use the notation $N_l(S,P)$ for the minimal number $m$ of points $x_i$ from the lattice generated by $P$, such that 
$$S\subset\cup_{i=1}^m x_i+P.$$

\item For $\delta\in (0,1)$, recall the notion of $\delta n-$sparse vectors 
$$Sparse(\delta n)=\{x\in\R^n:\,\#supp(x)\leq \delta n\},$$
where $supp(x)$ is the set of indecies of non-zero coordinates of $x.$

\item For $\alpha=(\alpha_1,...,\alpha_n)\in\R^n$ with $\alpha_i>0,$ we fix the notation $P_{\alpha}$ for the parallelepiped with sides $2\alpha_i$ and barycenter at the origin. That is,
\begin{equation}\label{Palpha}
P_{\alpha}=\{x\in\R^n:\,|x_i|\leq \alpha_i\}.
\end{equation}
For $\kappa>1,$ we consider the set of admissible parallelepipeds
\begin{equation}\label{Omegakappa}
\Omega_{\kappa}=\left\{\alpha\in\R^n:\,\alpha_i\in [0,1],\,\prod_{i=1}^n \alpha_i>\kappa^{-n}\right\}.
\end{equation}
\end{itemize}

Below we shall outline a few covering results.

\begin{lemma}\label{ball1}
For every $x\in\R^n$ there exists a finite set 
$$\mathcal{N}\subset \Z^n\cap (3\sqrt{n}B_2^n+x)$$ 
such that 
$$x+\sqrt{n}B_2^n\subset \mathcal{N}+B_{\infty}^n$$ 
and 
$$\#\mathcal{N}\leq C^n,$$ 
where $C$ is an absolute constant.
\end{lemma}
\begin{proof} For $x\in\R^n$, write $[x]:=([x_1],...,[x_n])$. Let
$$\mathcal{N}=[x]+\Z^n\cap 2\sqrt{n}B_2^n.$$
Since $x-[x]\in B_{\infty}^n\subset\sqrt{n}B_2^n,$ we have $\mathcal{N}\subset \Z^n\cap (3\sqrt{n}B_2^n+x)$. For the same reason, we have $x+\sqrt{n}B_2^n\subset \mathcal{N}+B_{\infty}^n$. 

Recall that the number of integer lattice points $z$ such that $\sum_{i=1}^n |z_i|\leq 2n$ is at most $\binom{3n}{n}\leq C^n$. In view of the fact that $2\sqrt{n}B_2^n\subset 2nB_1^n$, this implies that 
$$\#\left(\Z^n\cap 2\sqrt{n}B_2^n\right)\leq C^n.$$ 
Applying a lattice translation does not change anything, and we conclude that $\#\mathcal{N}\leq C^n,$ for some absolute constant $C>0.$ 

\end{proof}

As a corollary, we have

\begin{lemma}\label{coverbigside}
For any $\gamma\in (0,\sqrt{n}),$ for any $\epsilon>0,$ for any $x\in\R^n,$
$$N_l\left(x+\epsilon B_2^n,\frac{2\epsilon\gamma}{\sqrt{n}}B_{\infty}^n\right)\leq \max\left(\left(\frac{C_0}{\gamma}\right)^n, (C_1 \gamma)^{\frac{C_2 n}{\gamma^2}}\right).$$
Here $C_0,$ $C_1$ and $C_2$ are absolute constants.
\end{lemma}
\begin{proof} By scaling, we may assume without loss of generality that $\epsilon=1$. For $\gamma\in (0,2]$, Lemma \ref{ball1} implies that
$$N_l\left(x+B_2^n,\frac{2\gamma}{\sqrt{n}}B_{\infty}^n\right)\leq C^n N(B_2^n,2\gamma B_2^n)\leq \left(\frac{C_0}{\gamma}\right)^n,$$
where the last inequality is obtained by a standard volumetric argument, see, e.g., Vershynin \cite{Versh}.

Suppose now $\gamma\geq 2$. Recall that 
$$B_2^n\cap Sparse\left(\frac{n}{\gamma^2}\right)=\cup_{i=1}^{m}H_i\cap B_2^n,$$
where $H_i$ are subspaces of dimension $\frac{n}{\gamma^2}$, and $m=(C'\gamma)^{\frac{C_2 n}{\gamma^2}}$. Therefore, in view of Lemma \ref{ball1}, the set $B_2^n\cap Sparse\left(\frac{n}{\gamma^2}\right)$ admits a lattice covering by the translates of $\frac{2}{\sqrt{n}}B_{\infty}^n$ of cardinality $(C\gamma)^{\frac{C_2 n}{\gamma^2}}.$ Hence, for any $\gamma\geq 2,$ one has
\begin{equation}\label{sparsecover}
N_l\left(B_2^n\cap Sparse\left(\frac{n}{\gamma^2}\right),\,\frac{\gamma}{\sqrt{n}}B_{\infty}^n\right)\leq (C_1 \gamma)^{\frac{C_2 n}{\gamma^2}}.
\end{equation}

It remains to note, by the pigeonhole principle, that for any vector $x\in B_2^n$,
$$\#\left\{i=1,...,n:\, |x_i|\geq \frac{\gamma}{\sqrt{n}}\right\}\leq \frac{n}{\gamma^2}.$$
Therefore, 
$$N_l\left(x+B_2^n,\frac{2\gamma}{\sqrt{n}}B_{\infty}^n\right)\leq N_l\left(B_2^n,\frac{\gamma}{\sqrt{n}}B_{\infty}^n\right)=N_l\left(B_2^n\cap Sparse\left(\frac{n}{\gamma^2}\right),\frac{\gamma}{\sqrt{n}}B_{\infty}^n\right),$$ 
which, together with (\ref{sparsecover}), yields the lemma.
\end{proof}

The Lemma \ref{coverbigside} (in the non-lattice form) appears, e.g. in the work of Rebrova and Tikhomirov \cite{RebTikh}, where the ``pigeonhole principle'' reduction is used. It is crucial that the ball can be covered by a rather small (smaller than pure exponential) number of cubes with large diagonal. 

\medskip

Next, we will need the Lemma below, which bounds the covering number of the cube by parallelepipeds of large enough volume.

\begin{lemma}\label{paral}
For any $\kappa>1$ and for any $\alpha\in\Omega_{\kappa}$ there exists a lattice covering of $\frac{1}{2}B_{\infty}^n$ by $\kappa^n$ translates of $P_{\alpha}$. 
\end{lemma}
\begin{proof} 
Since $\alpha\in [0,1]^n$, there exists a (lattice) covering of $\frac{1}{2}B_{\infty}^n$ with $\prod_{i=1}^n [\frac{1}{\alpha_i}]$ translated copies of $P_{\alpha}$. 
It remains to note that
$$\prod_{i=1}^n [\frac{1}{\alpha_i}]\leq \prod_{i=1}^n \frac{1}{\alpha_i}\leq \kappa^n.$$
\end{proof}

We summarize the subsection with the following corollary of Lemma \ref{coverbigside} and Lemma \ref{paral}:

\begin{proposition}\label{covering-final}
Pick any $\gamma\in (1,\sqrt{n})$, any $\epsilon\in(0,\frac{1}{20\gamma})$, any $\kappa>1,$ and any $\alpha\in \Omega_{\kappa}.$ For any set $S\subset\sfe$, there exists a finite set 
$$\mathcal{N}\subset \frac{5}{4}B_2^n\setminus \frac{3}{4}B_2^n,$$ 
such that 
$$S\subset \mathcal{N}+\frac{4\epsilon\gamma}{\sqrt{n}}P_{\alpha},$$ 
and 
$$\#\mathcal{N}\leq N(S, \epsilon B_2^n)\cdot \kappa^n\cdot (C_1 \gamma)^{\frac{C_2 n}{\gamma^2}}.$$ 
\end{proposition}
\begin{proof}
Suppose 
$$S\subset \cup_{i=1}^m x_i+\epsilon B_2^n.$$
We apply Lemma \ref{coverbigside} to every $x_i+\epsilon B_2^n$, and get that 
$$N_l(S, \frac{2\epsilon\gamma}{\sqrt{n}}B_{\infty}^n)\leq N(S, \epsilon B_2^n)\cdot (C_1 \gamma)^{\frac{C_2 n}{\gamma^2}}.$$
Next,  apply Lemma \ref{paral} for each of the cubes $y_i+\frac{2\epsilon\gamma}{\sqrt{n}}B_{\infty}^n$ involved in the covering of $S$, yielding the net bound.

To assert that $\mathcal{N}\subset \frac{5}{4}B_2^n\setminus \frac{3}{4}B_2^n,$ it remains to observe that $\epsilon\gamma\leq\frac{1}{20}$.
\end{proof}



\section{Random rounding and the net construction}

\subsection{The random rounding}

Below we describe the idea of discretizing the unit sphere using a partially random net, to which we refer as ``random rounding''. It has been used a lot in computer science (see, e.g. Srinivasan \cite{Srin}, Kannan, Vempala \cite{KV}, Alon, Klartag \cite{KA}), and in the joint work with Klartag \cite{KLkold}.

\begin{definition}[random rounding]
Fix $\nu>0,$ $\kappa\geq 1$ and $\alpha\in\Omega_{\kappa}.$
	For $\xi \in \R^n$
	 consider a random vector $\tilde{\eta}^{\xi} \in \frac{\alpha_1\nu}{\sqrt{n}} \Z\times \frac{\alpha_2\nu}{\sqrt{n}} \Z\times...\times \frac{\alpha_n\nu}{\sqrt{n}} \Z$ with independent coordinates such that $|\xi_i - \tilde{\eta}^{\xi}_i| \leq \frac{\alpha_i\nu}{\sqrt{n}}$ 
	 with probability one, and $\E \tilde{\eta}^{\xi}=\xi$. Namely,  
	 for $i=1,\ldots,n$, writing $\xi_i = 
	 \frac{\alpha_i\nu}{\sqrt{n}} (k_i + p_i)$ for an integer $k_i$ and $p_i \in [0,1)$,
	\[
	\tilde{\eta}^{\xi}_i= 
	\begin{cases}
	\frac{\alpha_i\nu}{\sqrt{n}} k_i ,& \text{with probability}\,\,\, 1-p_i \\
	\frac{\alpha_i\nu}{\sqrt{n}} (k_i+1), & \text{with probability}\,\,\, p_i.
	\end{cases}
	\]
\end{definition}


\medskip

We shall need the following

\begin{lemma}[short vectors gravitate towards sparse]\label{sparserounding}
For an appropriately large $C_0>0,$ pick $n\geq C_0$. Fix $\gamma\in [C_0,\sqrt{n})$, and let $\kappa\geq 1$ be such that $\log\kappa\leq \frac{\log 2}{\gamma^{0.09}}.$ Fix any $\alpha\in \Omega_{\kappa}.$ Fix any $\xi\in B_2^n.$ With the random rounding $\tilde{\eta}^{\xi}$, taken with parameters $\nu=\gamma$ and $\alpha$, we have
$$P_{\eta}\left(\tilde{\eta}^{\xi}\in Sparse\left(\frac{n}{\gamma^{0.08}}\right)\right)\geq \frac{9}{10}.$$
Here $P_{\eta}$ stands for the probability taken with respect to the distribution of $\tilde{\eta}^{\xi}$.
\end{lemma}
\begin{proof} Consider a set of indices
$$\sigma_1=\left\{i:\alpha_i\geq\frac{1}{2}\right\}.$$
Recall that $\prod_{i=1}^n \alpha_i\geq \kappa^{-n}$ and that $\alpha_i\in [0,1]$. Observe, for any $M\geq 1,$ that 
$$\kappa^{-n}\leq \prod_{i=1}^n \alpha_i\leq \prod_{i:\, \alpha_i\leq \kappa^{-M}} \kappa^{-M},$$
and hence
$$\#\{i:\alpha_i\leq \kappa^{-M}\}\leq \frac{n}{M}.$$
By our assumption, $\log\kappa\leq \frac{\log 2}{\gamma^{0.09}},$ and therefore, plugging $M=\gamma^{0.09}$, we have
\begin{equation}\label{sigma1}
\#\sigma_1\geq n-\frac{n}{\gamma^{0.09}}.
\end{equation}
Next, consider the decomposition $\xi=k+p,$ where $k\in \frac{\alpha_1\gamma}{\sqrt{n}} \Z\times \frac{\alpha_2\gamma}{\sqrt{n}} \Z\times...\times \frac{\alpha_n\gamma}{\sqrt{n}} \Z$ and $p\in \frac{\gamma}{\sqrt{n}}P_{\alpha}$. Let
$$\sigma_2=\{i:\,k_i=0\}.$$
Note that $1\geq |\xi|^2,$ and hence 
$$\#\sigma_2^c\cap \sigma_1=\#\left\{i\in\sigma_1:\, |\xi_i|\geq \frac{\alpha_i\gamma}{\sqrt{n}}\right\}\leq \#\left\{i\in\sigma_1:\, |\xi_i|\geq \frac{\gamma}{2\sqrt{n}}\right\}\leq \frac{4n}{\gamma^2}.$$
Therefore, in view of the fact that $\sigma_2^c\subset (\sigma_2^c\cap \sigma_1)\cup \sigma_1^c$, we have
\begin{equation}\label{sigma2}
\#\sigma_2\geq n-\frac{4n}{\gamma^2}-\frac{n}{\gamma^{0.09}}.
\end{equation}
Next, for a large $A\in [4, \gamma]$, consider 
$$\sigma_3=\left\{i:\,|\xi_i|\leq \frac{\gamma}{2A\sqrt{n}}\right\}.$$
Note, as before, that
\begin{equation}\label{sigma3}
\#\sigma_3\geq n-\frac{4nA^2}{\gamma^2}
\end{equation}
Finally, let 
$$\sigma=\sigma_1\cap\sigma_2\cap \sigma_3.$$
Observe that 
\begin{equation}\label{sigma}
\#\sigma\geq n-\frac{8nA^2}{\gamma^2}-\frac{2n}{\gamma^{0.09}}.
\end{equation}
Note, by construction, and by the definition of $\tilde{\eta}^{\xi}$, that for every $i\in\sigma,$ we have
\begin{equation}\label{prob-A}
P(\tilde{\eta}^{\xi}_i=0)\geq 1-\frac{1}{A}.
\end{equation}
Let $\omega>0$ be chosen later; suppose that 
\begin{equation}\label{bnd1}
\frac{n}{\gamma^{\omega}}\geq 2\left(\frac{8nA^2}{\gamma^2}+\frac{2n}{\gamma^{0.09}}\right).
\end{equation}
Recall, by definition of sparsity:
$$P\left(\tilde{\eta}^{\xi}\not\in Sparse\left(\frac{n}{\gamma^{\omega}}\right)\right)=$$$$P\left(\exists \sigma_4\subset\{1,...,n\},\,\#\sigma_4> \frac{n}{\gamma^{\omega}}:\,\forall i\in\sigma_4:\,\tilde{\eta}_i^{\xi}\neq 0\right).$$
We use (\ref{bnd1}) and estimate it from above by
$$P\left(\exists \sigma_5\subset\sigma,\,\#\sigma_5=\frac{n}{2\gamma^{\omega}}:\,\forall i\in\sigma_5:\,\tilde{\eta_i}^{\xi}\neq 0\right)\leq$$ 
\begin{equation}\label{toest}
\left(C\gamma^{\omega}\right)^{\frac{2en}{\gamma^{\omega}}}\cdot\left(P(\tilde{\eta}_i^{\xi}\neq 0\,|i\in\sigma)\right)^{\frac{n}{2\gamma^{\omega}}}.
\end{equation}
In view of (\ref{prob-A}), we estimate (\ref{toest}) from above with
\begin{equation}\label{finalest}
\left(C\gamma^{\omega}\right)^{\frac{2en}{\gamma^{\omega}}} A^{-\frac{n}{2\gamma^{\omega}}}\leq \frac{1}{10},
\end{equation}
provided that $C_0$ is chosen large enough and that
\begin{equation}\label{bnd2}
A= C'\gamma^{4e\omega},
\end{equation}
with an appropriate constant $C'$. In order to satisfy (\ref{bnd1}) and (\ref{bnd2}), we see that it is enough to select
$$\omega\leq \min\left(\frac{2}{1+8e}, 0.09\right),$$
again, provided $C_0>0$ is large enough. We select $\omega=0.08.$
\end{proof}

We shall need also a notion of adapted random rounding: the idea is to consider a fixed covering of a set $S$ by euclidean balls, and round each point in $S$ with respect to the lattice associated with the center of the (a-priori fixed) euclidean ball in which it lies.

\begin{definition}[adapted random rounding]
Consider a set $S\subset\R^n$ and fix a covering $\mathcal{F}$ by $\epsilon B_2^n,$ with centers at $S$:
$$S\subset \cup_{j=1}^m x_j+\epsilon B_2^n,\,\,\,\,x_j\in S.$$
For every $\xi\in S$, select and fix any $j=j(\xi)$ such that $\xi\in x_j+\epsilon B_2^n$ (there a-priori could be several such $j$ but we fix one for each $\xi$). Given $\kappa,$ $\alpha,$ $\nu$ from the definition of random rounding, consider an ``adapted rounding'' on $S$ with respect to $\mathcal{F}$, given by  
$$\eta^{\xi}=\tilde{\eta}^{\xi-x_j}+x_j.$$
\end{definition}

Next, we formulate the following corollary of Proposition \ref{covering-final} and Lemma \ref{sparserounding}:

\begin{corollary}\label{maincor-rounding}
For an appropriately large $C_0>0,$ pick $n\geq C_0$. Fix $\gamma\in [C_0,\sqrt{n})$, and let $\kappa\geq 1$. Fix any $\alpha\in \Omega_{\kappa}.$ Fix $\epsilon\in (0,\frac{1}{20\gamma}).$ Consider any $S\subset\sfe,$ and fix $\mathcal{F}\subset S$ such that
$$S\subset \cup_{x_j\in\mathcal{F}} x_j+\epsilon B_2^n,$$
and $\mathcal{F}$ is optimal, that is, $\#\mathcal{F}=N(S, \epsilon B_2^n).$

Consider the adapted random rounding $\eta^{\xi}$ on $S$, taken with parameters $\nu=\epsilon \gamma,$ and $\alpha$.

Then there exists a net $\mathcal{N}\subset S+4\epsilon\gamma B_2^n$, with
        \[
	\#\mathcal{N}\leq 
	\begin{cases} N(S, \epsilon B_2^n)\cdot(C_1 \gamma)^{\frac{C_2 n}{\gamma^{0.08}}},& \text{if}\,\,\, \log\kappa\leq \frac{\log 2}{\gamma^{0.09}},\\
	N(S, \epsilon B_2^n)\cdot(C\kappa)^n,& \text{if}\,\,\, \log\kappa\geq \frac{\log 2}{\gamma^{0.09}},\\
	\end{cases}
	\]	
such that whenever for $M>0$ and a measurable function $F:\R^n\rightarrow \R^+$ we have, for every $\xi\in S,$ that
$$\mathbb{E}_{\eta} F(\xi-\eta^{\xi})\leq M,$$
then for every $x\in S$ there exists $y\in\mathcal{N}$ with
$$F(x-y)\leq 2M.$$
\end{corollary}
\begin{proof} \textbf{Case 1.} Suppose $\log\kappa\geq \frac{\log 2}{\gamma^{0.09}}.$ By Proposition \ref{covering-final}, there exists a net 
$$\mathcal{K}\subset \frac{5}{4}B_2^n\setminus \frac{3}{4}B_2^n$$ 
such that
$$S\subset \mathcal{K}+\frac{\epsilon \gamma}{\sqrt{n}}P_{\alpha},$$
and
$$\#\mathcal{K}\leq N(S, \epsilon B_2^n)\cdot \kappa^n\cdot (C_1 \gamma)^{\frac{C_2 n}{\gamma^2}},$$
and such that $\eta^{\xi}$ takes values in 
$$\mathcal{N}=\mathcal{K}\pm\frac{\alpha_1\epsilon\gamma}{\sqrt{n}}e_1\pm...\pm \frac{\alpha_n\epsilon\gamma}{\sqrt{n}}e_n\subset \frac{3}{2}B_2^n\setminus \frac{1}{2}B_2^n,$$ 
where the inclusion holds since $\epsilon\leq \frac{1}{20\gamma}.$ Namely, $\mathcal{N}$ is the collection of vertices of the parallelepipeds covering $S$.

Since $\mathbb{E}_{\eta} F(\xi-\tilde{\eta}^{\xi})\leq M,$ there exists an $\eta\in\mathcal{N}$, such that $F(\xi-\eta)\leq M.$ Note that each parallelepiped has $2^n$ vertices. The estimate on the cardinality of the net thus follows from the inequality $\#\mathcal{N}\leq 2^n\cdot \#\mathcal{K},$ and in view of the fact that $(C_1 \gamma)^{\frac{C_2 n}{\gamma^2}}\leq C^n.$

\medskip

\textbf{Case 2.} Next, suppose $\log\kappa\leq \frac{\log 2}{\gamma^{0.09}}.$ Then we may apply Lemma \ref{sparserounding}. Fix the lattice net $\mathcal{L}$ in $\R^n$ generated by $\frac{\epsilon\gamma}{\sqrt{n}}P_{\alpha}.$ Pick any $x_j\in \mathcal{F}$ -- the center of a euclidean ball from our fixed euclidean net $\mathcal{F}$. Consider the set 
$$\mathcal{N}_j=x_j+Sparse\left(\frac{n}{\gamma^{0.08}}\right)\bigcap 3\gamma\epsilon B_2^n\bigcap \mathcal{L}.$$
Note that 
$$\#\mathcal{N}_j\leq \kappa^n\cdot e^{\frac{cn}{\gamma^{0.08}}}\cdot (C''_1 \gamma)^{\frac{C''_2 n}{\gamma^{0.08}}}\leq (C'_1 \gamma)^{\frac{C'_2 n}{\gamma^{0.08}}},$$
since $\kappa^n\leq 2^{\frac{n}{\gamma^{0.09}}}$ in the current case. Consider 
$$\mathcal{N}=\cup_{x_j\in \mathcal{F}} \mathcal{N}_j.$$ 
We deduce that 
$$\#\mathcal{N}\leq (C'_1 \gamma)^{\frac{C'_2 n}{\gamma^{0.08}}}\cdot\#\mathcal{F}\leq N(S, \epsilon B_2^n)\cdot(C_1 \gamma)^{\frac{C_2 n}{\gamma^{0.08}}}.$$

Fix $\xi\in S$. By definition of the adapted random rounding, and in view of Lemma \ref{sparserounding}, the random vector $\eta^{\xi}$ (taken with parameters $\nu=\epsilon\gamma$, $\alpha,$ $\kappa$, with respect to $\mathcal{F}$) takes values in $\mathcal{N}$ with probability at least $\frac{9}{10}$. Further, by our assumption, with probability at least $\frac{1}{2},$ we have $F(\xi-\eta^{\xi})\leq 2M.$ Therefore, there exists an $\eta\in\mathcal{N}$ satisfying $F(\xi-\eta)\leq 2M,$ and the proof is done.  
\end{proof}

\medskip





\medskip

\subsection{Comparison via Hilbert-Schmidt}

We begin by formulating an analogue of Lemma 5.1 from \cite{KLkold}:

\begin{lemma}[comparison via Hilbert-Schmidt]\label{HS}
Let $\epsilon\in (0,\frac{1}{20}).$ There exists a collection of points $\mathcal{F}\subset \frac{3}{2}B_2^n\setminus \frac{1}{2}B_2^n$ with $\#\mathcal{F}\leq (\frac{C}{\epsilon})^{n-1}$ such that for any (deterministic) matrix $A:\R^n\rightarrow \R^N$, for every $\xi\in\sfe$ there exists an $\eta\in\mathcal{F}$ satisfying
\begin{equation}\label{HS-comp}
|A(\eta-\xi)|\leq \frac{\epsilon}{\sqrt{n}}||A||_{HS}.
\end{equation}
Here $C$ is an absolute constant.
\end{lemma}
\begin{proof} Recall that $|Ax|^2=\sum_{i=1}^N \langle A^Te_i, x\rangle^2$. Consider the random rounding with $\kappa=\alpha_1=...=\alpha_n=1$ and $\nu=2\epsilon$. For every vector $g\in\R^n$, using the fact that $\mathbb{E}_{\eta}\langle \xi-\eta^{\xi},g\rangle=0$, as well as the fact that $||\xi-\eta^{\xi}||_{\infty}\leq \frac{\epsilon}{\sqrt{n}}$, we deduce
\begin{equation}\label{exp}
\mathbb{E}_{\eta}|\langle \eta^{\xi}, g\rangle-\langle \xi, g\rangle|^2\leq\frac{\epsilon^2|g|^2}{n}.
\end{equation}
By applying (\ref{exp}) to the rows of $A$ and summing up, we get
\begin{equation}\label{expect-bound}
\E_{\eta}|A(\eta^{\xi}-\xi)|^2=\E_{\eta}\sum_{i=1}^N \langle A^Te_i, \eta^{\xi}-\xi\rangle^2\leq \left(\frac{\epsilon}{\sqrt{n}}||A||_{HS}\right)^2.
\end{equation}

Recall (see, e.g. Rudelson, Vershynin \cite{RudVer-general}) that $N(\sfe, \epsilon B_2^n)\leq (\frac{C'}{\epsilon})^{n-1}$; applying Corollary \ref{maincor-rounding} with $\gamma=2,$ $\kappa=\alpha_1=...=\alpha_n=1$ and $F(x)=|Ax|^2$ finishes the proof, with the appropriate choice of constants.


\end{proof}

A statement similar to Lemma \ref{HS} was recently independently proved and used by Lytova and Tikhomirov \cite{LytTikh}.

\begin{remark} One may note that the inequality $\E|\xi_i-\eta_i^{\xi}|^2\leq \frac{\epsilon^2}{n}$, which was used in the proof above, is not sharp, in fact: indeed, the variance of a $p-$Bernoulli random variable is $p(1-p)$, thus the conclusion in Lemma \ref{HS} could be strengthened on a subset of the sphere of vectors which are close to a scaled lattice. 
\end{remark}

\subsection{Refinement of the Hilbert-Schmidt norm}

Lemma \ref{HS} shows that there exists a net of cardinality $C^n$, such that for any random matrix $A:\R^n\rightarrow\R^N$, with probability at least $1-P(||A||_{HS}^2\geq 10\mathbb{E}||A||_{HS}^2)$, one has (\ref{HS-comp}) with $\sqrt{\E||A||^2_{HS}}$ in place of $||A||_{HS}$. However, such probability estimate shall be unsatisfactory for our purposes: indeed, assuming only the bounded second moments, the best estimate for the said probability we could have, is that entailed by Markov's inequality, i.e. $\frac{9}{10}$. In order to overcome this issue, we employ the idea of Rebrova and Tikhomirov \cite{RebTikh}: in place of the covering by cubes, we consider a covering by paralelleipeds of sufficiently large volume. To this end, we formulate our key definition.

\begin{definition}\label{ourB}
Fix $\kappa>1$. For an $N\times n$ matrix $A$, define 
$$\B_{\kappa}(A):=\min_{\alpha\in\Omega_{\kappa}} \sum_{i=1}^n \alpha_i^2 |Ae_i|^2.$$
Recall that $\Omega_{\kappa}$ was defined in (\ref{Omegakappa}).
\end{definition}

$\B_{\kappa}(A)$ should be thought of as a certain averaging process on the lengths of the columns of $A$. As we shall show, the large deviation properties of $\B_{\kappa}(A)$ are significantly better than those of the usual Hilbert-Schmidt norm. At the same time, it turns out that we can use $\B_{\kappa}(A)$ as a measure of comparison for the values of $|Ax|$ to the corresponding values on a net of relatively small size. 

First, we refine the result of Lemma \ref{HS}.

\begin{lemma}\label{compviaB}
Let $\gamma\in (1,\sqrt{n})$, $\epsilon\in(0,\frac{1}{20 \gamma})$, $\kappa>1$. Pick any $\alpha\in\Omega_{\kappa}$. Let $A$ be any $N\times n$ matrix. Consider any $S\subset\sfe$.

There exists a collection of points $\mathcal{F}_{\alpha}\subset S+4\epsilon\gamma B_2^n$ with 
\[
	\#\mathcal{\mathcal{F}_{\alpha}}\leq 
	\begin{cases} N(S, \epsilon B_2^n)\cdot(C_1 \gamma)^{\frac{C_2 n}{\gamma^{0.08}}},& \text{if}\,\,\, \log\kappa\leq \frac{\log 2}{\gamma^{0.09}},\\
	N(S, \epsilon B_2^n)\cdot(C\kappa)^n,& \text{if}\,\,\, \log\kappa\geq \frac{\log 2}{\gamma^{0.09}},\\
	\end{cases}
	\]	
	such that for every $\xi\in S$ there exists an $\eta\in\mathcal{F}_{\alpha}$ satisfying
\begin{equation}\label{B-comp}
|A(\eta-\xi)|\leq \frac{2\gamma\epsilon}{\sqrt{n}}\sqrt{\sum_{i=1}^n \alpha_i^2 |Ae_i|^2}.
\end{equation}
Here $C, C_1, C_2$ are absolute constants.
\end{lemma}
\begin{proof} Consider the random rounding $\eta^{\xi}$ with $\nu=\epsilon\gamma$ and $\alpha$. As before, we have, for any vector $g=(g_1,...,g_n)\in\R^n,$
\begin{equation}\label{improved}
\E_{\eta}|\langle \eta^{\xi}-\xi, g\rangle|^2 \leq \frac{\epsilon^2\gamma^2\sum_{i=1}^n \alpha_i^2g_i^2}{n}.
\end{equation}
Therefore, summing up, we get
\begin{equation}\label{expect-bound-B}
\E_{\eta}|A(\eta^{\xi}-\xi)|^2\leq \frac{\epsilon^2\gamma^2}{n}\sum_{i=1}^n \alpha_i^2 |Ae_i|^2.
\end{equation}
The Lemma hence follows from Corollary \ref{maincor-rounding}, with the function $F(y)=|Ay|^2$.
\end{proof}

Observe that Lemma \ref{compviaB} implies the following (so far, unsatisfactory) corollary:

\begin{corollary}[comparison via $\B(A)$]\label{compviaB-cor}
Let $\gamma\in (1,\sqrt{n})$, $\epsilon\in(0,\frac{1}{10 \gamma})$, $\kappa>1$. Let $A$ be any $N\times n$ matrix. Consider any $S\subset\sfe$.

There exists a collection of points $\mathcal{F}\subset \frac{3}{2}B_2^n\setminus \frac{1}{2}B_2^n$ with cardinality bound as in Lemma \ref{compviaB}, such that for every $\xi\in\sfe$ there exists an $\eta\in\mathcal{F}$ satisfying
\begin{equation}\label{B-comp}
|A(\eta-\xi)|\leq \frac{2\gamma\epsilon}{\sqrt{n}}\sqrt{\B_{\kappa}(A)}.
\end{equation}
\end{corollary}

Next, we would like to switch the quantifiers in the previous statement: in place of the net that depends on the matrix, we need to have a fixed net, which serves all matrices. For that purpose we shall consider a net on the set of admissible nets.

\begin{lemma}[nets on nets]\label{netsonnets}
There exist absolute constants $C, C', C''>0$ such that for any $\kappa>1$ and $\mu\in (0,\sqrt{n})$ there exists a collection $\mathcal{F}\subset \Omega_{\kappa^{1+\mu}}$ of cardinality 
\begin{equation}\label{mucard}
\max\left(\left(\frac{C}{\mu}\right)^{n-1}, (C'\mu)^{\frac{C''n}{\mu^2}}\right),
\end{equation}
such that for any $\alpha\in\Omega_{\kappa}$ there exists a $\beta\in \mathcal{F}$ such that for all $i=1,...,n$ we have $\alpha_i^2\geq\beta_i^2.$ 

In particular, for any $N\times n$ matrix $A$, we have
$$\B_{\kappa}(A)\geq \min_{\beta\in\mathcal{F}}\sum_{i=1}^n \beta_i^2 |Ae_i|^2.$$
\end{lemma} 
\begin{proof} Consider a transformation $T:\R^n \rightarrow \R^n$ given by
$$T\alpha=\left(...,\sqrt{\frac{\log |\frac{1}{\alpha_i}|}{n\log \kappa}},...\right).$$
Denote $B=B_2^n\cap\{x_i\geq 0\,\forall i=1,...,n\}$. Then, by definition of $\Omega_{\kappa}$ we have
$$T\Omega_{\kappa}=B,$$
and 
$$T^{-1}\left((1+\mu)B\right)=\Omega_{\kappa^{1+\mu}}.$$
Note that this mapping is a bijection on $\Omega_{\kappa}$ as well as on $\Omega_{\kappa^{1+\mu}}$.

Consider a lattice covering $\mathcal{N}$ of the boundary of $B$ with translates of $\frac{\mu}{\sqrt{n}}B_{\infty}^n$. In each cube $x+\frac{\mu}{\sqrt{n}}B_{\infty}^n$ from this covering, pick such a vertex $v(x)$ that for all $y\in x+\frac{\mu}{\sqrt{n}}B_{\infty}^n$, and for all $i=1,...,n$, one has $y_i\leq v(x)_i$. Define $\mathcal{S}=\{v(x):\,x\in\mathcal{N}\}$. Note that $\mathcal{S}\subset (1+\mu)B,$ and that
$$
\#\mathcal{S}=\#\mathcal{N}\leq\min\left(\left(\frac{C}{\mu}\right)^{n-1}, (C'\mu)^{\frac{C''n}{\mu^2}}\right),
$$
where the last inequality follows from Lemma \ref{coverbigside} (note that the power $n-1$ comes from the fact that we are covering the sphere rather than the ball).


Let $\mathcal{F}=T^{-1}\mathcal{S}\subset \Omega_{\kappa^{1+\mu}}$. For every $\alpha\in\Omega_{\kappa}$ let $a=T\alpha\in B.$ Then take the $b\in \mathcal{S}\subset (1+\mu)B$ such that $a_i^2\leq b_i^2;$ consider $\beta\in \mathcal{F}$ defined as $\beta=T^{-1}b.$ Since $T$ is coordinate-vise decreasing, we have, for all $i\in\{1,...,n\}$, the inequality $\alpha_i^2\geq \beta_i^2,$ as desired. 
\end{proof}


Finally, we deduce the result about a net which serves all deterministic matrices.

\begin{thm}[sharp net for deterministic matrices]\label{main-net}
Fix the dimension $n\in\N.$ Consider any $S\subset\sfe.$ Pick any $\gamma\in(2,\sqrt{n})$, $\epsilon\in(0,\frac{1}{10 \gamma})$, $\kappa>1.$ 
There exists a (deterministic) net $\mathcal{N}\subset S+4\epsilon\gamma B_2^n$, with 
\[
	\#\mathcal{N}\leq 
	\begin{cases} N(S, \epsilon B_2^n)\cdot(C_1 \gamma)^{\frac{C_2 n}{\gamma^{0.08}}},& \text{if}\,\,\, \log\kappa\leq \frac{\log 2}{\gamma^{0.09}},\\
	N(S, \epsilon B_2^n)\cdot(C\kappa)^n(\log\kappa)^{n-1},& \text{if}\,\,\, \log\kappa\geq \frac{\log 2}{\gamma^{0.09}},\\
	\end{cases}
	\]	
such that for every $N\in\mathbb{N}$ and for every (deterministic) $N\times n$ matrix $A$, the following holds: for every $x\in S$ there exists $y\in\mathcal{N}$ such that 
\begin{equation}\label{netmain1}
|A(x-y)|\leq\frac{2\gamma\epsilon}{\sqrt{n}}\sqrt{\B_{\kappa}(A)}.
\end{equation}
Here $C, C_0, C_1, C_2$ are absolute constants.
\end{thm}
\begin{proof}

Let $\mu=\min(\frac{1}{\log\kappa}, c\gamma^{0.09})$. Consider $\mathcal{F}\subset \Omega_{\kappa^{1+\mu}}$, given to us by Lemma \ref{netsonnets}, with
\begin{equation}\label{size1}
\#\mathcal{F}\leq \min\left((C\log\kappa)^{n-1}, (c\gamma)^{\frac{c'n}{\gamma^{0.18}}}\right),
\end{equation}
and
\begin{equation}\label{inf}
\min_{\beta\in\mathcal{F}} \sum_{i=1}^n \beta_i^2|Ae_i|^2\leq \B_{\kappa}(A).
\end{equation}
For each $\beta\in\mathcal{F}$, consider the net $\mathcal{N}_{\beta}$ constructed in Lemma \ref{compviaB}. 
By Lemma \ref{compviaB} we conclude, that for any $N\in\N,$ for any $N\times n$ matrix $A$, for any $x\in\sfe$ there exists $y\in \mathcal{N}_{\beta}$ with
\begin{equation}\label{key1}
|A(x-y)|\leq 2\left(\frac{\epsilon^2\gamma^2}{n}\sum_{i=1}^n \beta_i^2 |Ae_i|^2\right)^{\frac{1}{2}}.
\end{equation}
Next, let $\mathcal{N}=\cup_{\beta\in\mathcal{F}}\mathcal{N}_{\beta}$. Then, for any $N\times n$ matrix $A$, for any $x\in\sfe$ there exists $y\in \mathcal{N}$ such that
\begin{equation}\label{key2}
|A(x-y)|\leq 2\left(\frac{\epsilon^2\gamma^2}{n}\min_{\beta\in\mathcal{F}}\sum_{i=1}^n \beta_i^2 |Ae_i|^2\right)^{\frac{1}{2}}.
\end{equation}
Combining (\ref{inf}) and (\ref{key2}), we conclude (\ref{netmain1}). It remains to observe, in view of the fact that $k^{\frac{1}{\log\kappa}}=e$:
\[
	\#\mathcal{N}\leq 
	\begin{cases} N(S, \epsilon B_2^n)\cdot(C_1 \gamma)^{\frac{C_2 n}{\gamma^{0.08}}},& \text{if}\,\,\, \log\kappa\leq \frac{\log 2}{\gamma^{0.09}},\\
	N(S, \epsilon B_2^n)\cdot(C\kappa)^n(\log\kappa)^{n-1},& \text{if}\,\,\, \log\kappa\geq \frac{\log 2}{\gamma^{0.09}}.\\
	\end{cases}
	\]
\end{proof}

\medskip

\subsection{Large deviation of $\B_{\kappa}(A)$.}

Theorem \ref{main-net} is a statement about deterministic matrices only; it reduces the estimate of probability with which for a random matrix $A$ one may find a sharp net, to the large deviation properties of the random variable $\B_{\kappa}(A)$. Philosophically, $\B_{\kappa}(A)$ is a ``minimum'', and minima have good large deviation properties; see, e.g. Lugosi, Mendelson \cite{LugMen}, where a similar idea in order to improve deviation behavior is used in the context of statistical mean approximation.

\begin{lemma}\label{ldpB}
Let $A$ be a random matrix with independent columns. Pick any $\kappa>1$, $p>0$ and $s>0$. Then
$$P\left(\B_{\kappa}(A)\geq 2s\sum_{i=1}^n \left(\mathbb{E}|Ae_i|^{2p}\right)^{\frac{1}{p}}\right)\leq \kappa^{-2pn}\left(1+\frac{1}{s^p}\right)^n.$$
\end{lemma}
\begin{proof} Assume that $\sum_{i=1}^n \left(\mathbb{E}|Ae_i|^{2p}\right)^{\frac{1}{p}})<\infty$; otherwise, the statement is self-redundant. 

Denote $Y_i=|Ae_i|$. If $\B_{\kappa}(A)\geq 2s\sum_{i=1}^n \left(\mathbb{E}|Ae_i|^{2p}\right)^{\frac{1}{p}}$, then for any collection $\alpha_1,...,\alpha_n\in[0,1]$, either 
$$\sum_{i=1}^n \alpha_i^2 Y_i^2\geq 2s\sum_{i=1}^n \left(\mathbb{E}Y_i^{2p}\right)^{\frac{1}{p}},$$ 
or 
$$\prod_{i=1}^n \alpha_i<\kappa^{-n}.$$ 

If there was no assumption $\alpha_i\in[0,1]$, then the sequence, minimizing $\sum \alpha_i^2 Y_i^2$, for a fixed collection of $Y_i$, under the constraint $\prod_i \alpha_i=\kappa^{-n},$ would be $\alpha_i=\frac{c_0}{Y_i}$, with $c_0=\kappa^{-1}\left(\prod_{i=1}^n Y_i\right)^{1/n}$. This hints us to consider a collection of random variables 
$$\alpha_i^2=\min\left(1,\frac{s\left(\mathbb{E}Y_i^{2p}\right)^{\frac{1}{p}}}{Y_i^2}\right).$$ 
We estimate 
$$P\left(\B_{\kappa}(A)\geq 2s\sum_{i=1}^n \left(\mathbb{E}Y_i^{2p}\right)^{\frac{1}{p}}\right)\leq $$
$$P\left(\sum_{i=1}^n \min\left(1,\frac{s\left(\mathbb{E}Y_i^{2p}\right)^{\frac{1}{p}}}{Y_i^2}\right)Y_{i}^2 \geq 2s\sum_{i=1}^n \left(\mathbb{E}Y_i^{2p}\right)^{\frac{1}{p}}\right)+$$
$$P\left(\prod_{i=1}^n \min\left(1,\frac{s\left(\mathbb{E}Y_i^{2p}\right)^{\frac{1}{p}}}{Y_i^2}\right)<\kappa^{-2n}\right)=:P_1+P_2.$$

Note that 
$$P_1\leq P\left(s\sum_{i=1}^n \left(\mathbb{E}Y_i^{2p}\right)^{\frac{1}{p}}\geq 2s\sum_{i=1}^n \left(\mathbb{E}Y_i^{2p}\right)^{\frac{1}{p}}\right)=0.$$ 

Next, by Markov's inequality, for any $p>0$ we have
$$P_2\leq \kappa^{-2pn}\prod_{i=1}^n\mathbb{E}\frac{1}{\min\left(1,\frac{s\left(\mathbb{E}Y_i^{2p}\right)^{\frac{1}{p}}}{Y_i^2}\right)^p}\leq \kappa^{-2pn}\left(1+\frac{1}{s^p}\right)^n,$$
where in the last inequality we used independence of columns, and the fact that 
$$\mathbb{E}\frac{1}{\min\left(1,\frac{s\left(\mathbb{E}Y_i^{2p}\right)^{\frac{1}{p}}}{Y_i^2}\right)^p}\leq \E\left(1+\frac{Y_i^{2p}}{s^p\mathbb{E}Y_i^{2p}}\right)=1+\frac{1}{s^p}.$$
Here we also used that for a positive $p$ and a positive $a,$ one has $a^p>1$ if and only if $a>1.$
This finishes the proof.
\end{proof}

\begin{remark}\label{remarkB}
Under different assumptions on the random matrix one might hope to deduce different estimates on the large deviation of $\B_{\kappa}(A)$. For example, already the bounded concentration function assumption yields certain improvements to the bound of the Lemma \ref{ldpB} in the case $\kappa=10$: one may compare $\B_{\kappa}(A)$ to, say, $\B_{s\kappa}(A)$ for large $s$, and deduce that either the large deviation of $\B_{s\kappa}$ entails some large deviation of $\B_{\kappa}$, or $\B_{\kappa}$ is much smaller than $\B_{s\kappa}$, and hence the lengths of columns of $A$ are highly irregular. It is, however, not clear if the stronger assumption of small ball probability for columns could lead to a more substantial improvement. An affirmative answer could, in particular, slightly improve Theorem \ref{main}, since in Section 8 the net Theorem \ref{keytheoremnets} is applied to a matrix which w.h.p. possesses a strong small ball probability for columns.

In view of Theorem \ref{main-net} (which itself involves no randomness), any estimate on the large deviation of $\B_{\kappa}(A)$ would have rather straightforward implications for the smallest singular value estimates.
\end{remark}

\medskip

\subsection{Proof of Theorem \ref{keytheoremnets}.} The statement follows immediately from Theorem \ref{main-net} and Lemma \ref{ldpB}.

\medskip

We conclude this section by formulating two corollaries of Theorem \ref{keytheoremnets}, in two different regimes, which we shall need. Firstly, we consider the regime in which the ``short vectors gravitate towards sparse'' lemma plays a crucial role, as we construct a net of size $e^{\mu n}$, with arbitrarily small constant $\mu$, while only spoiling the comparison constant polynomially.

\begin{cor}[regime of small $\kappa$ and sparsity]\label{cor1}
Fix $n,N\in\N,$ and fix any $p>0$. Consider any $S\subset\sfe.$ For any $\mu\in (0,1),$ and for every $\epsilon\in (0, \mu^{c_0})$ (with some absolute constant $c_0>0$), there exists a (deterministic) net $\mathcal{N}\subset S+4\epsilon\gamma B_2^n$, with 
$$\#\mathcal{N}\leq N(S,\epsilon B_2^n)\cdot e^{\mu n},$$ 
and there exist positive constants $C_1(\mu)$ which depends (polynomially) only on $\mu$, and $C_2(\mu, p)$, which depends only on $p$ and $\mu,$ such that for every random $N\times n$ matrix $A$ with independent columns, with probability at least 
$$1-e^{-C_1(\mu)pn},$$
for every $x\in S$ there exists $y\in\mathcal{N}$ so that 
\begin{equation}\label{netmain}
|A(x-y)|\leq \frac{C_2(\mu, p)\epsilon}{\sqrt{n}}\sqrt{\sum_{i=1}^n \left(\E|Ae_i|^{2p}\right)^{\frac{1}{p}}}.
\end{equation}
\end{cor}
\begin{proof} In Theorem \ref{keytheoremnets}, let $\gamma$ be sufficiently large so that $(C_1 \gamma)^{\frac{C_2 n}{\gamma^{0.08}}}\leq e^{\mu n}.$ Then for every positive $\epsilon\leq\mu^{c_0},$ there exists a net $\mathcal{N}$ satisfying the required cardinality bound. Further, let $\kappa=2^{\frac{1}{\gamma^{0.09}}},$ and let $s=\left(\frac{1}{p\log\kappa }\right)^{\frac{1}{p}}$. Then with probability 
$$1-\kappa^{-2pn}(1+s^{-p})^n=1-e^{-C_1(\mu)pn},$$ 
the desired net comparison holds with $C_2(\mu,p)=2\gamma \sqrt{s}$.
\end{proof}

\medskip

Next, we outline a result in which the net cardinality is estimated, roughly, up to $10^n,$  but the probability that the efficient net exists can be arbitrarily close to $1.$
\begin{cor}[regime of large $\kappa$]\label{cor2}
Fix $n,N\in\N.$ Consider any $S\subset\sfe.$ Pick any $\epsilon\in (0,\frac{1}{20})$, $\kappa>1,$ and $p>0$.
 
There exists a (deterministic) net $\mathcal{N}\subset S+4\epsilon\gamma B_2^n$, with 
$$\#\mathcal{N}\leq N(S,\epsilon B_2^n)\cdot (C\kappa)^n(\log\kappa)^{n-1},$$ 
such that for every random $N\times n$ matrix $A$ with independent columns, with probability at least 
$$1-(C'\kappa)^{-2pn},$$
for every $x\in S$ there exists $y\in\mathcal{N}$ such that 
\begin{equation}\label{netmain}
|A(x-y)|\leq \frac{C_1\epsilon}{\sqrt{n}}\sqrt{\sum_{i=1}^n \left(\E|Ae_i|^{2p}\right)^{\frac{1}{p}}}.
\end{equation}
\end{cor}
\begin{proof} In Theorem \ref{keytheoremnets}, let $\gamma=2$ and $s=2,$ and select the constants appropriately. 
\end{proof}

\section{Small-ball estimates: survey of the known results which we shall use.}

In this section, we recall several known results about small ball probability, in particular powerful estimates via LCD, derived by Rudelson and Vershynin \cite{RudVer-square}, \cite{RudVer-general}, \cite{RudVer-delocalization}. We emphasize that this section is a survey, and contains no novelty. 

We shall need ``the tensorization Lemma'' observed by Rudelson and Vershynin \cite{RudVer-square}.

\begin{lemma}[tensorization lemma, Rudelson-Vershynin, \cite{RudVer-square}]\label{tensorization}
Fix $\epsilon>0.$ Suppose $Y_1,...,Y_M$ are non-negative independent random variables, and assume that for each of them we have
$$P(Y_i\leq \epsilon)\leq K\epsilon.$$
Then
$$P\left(\sum_{i=1}^M Y_i^2\leq M\epsilon^2\right)\leq (CK\epsilon)^M.$$
\end{lemma}

The elementary proof of Lemma \ref{tensorization} can be found in \cite{RudVer-square}, where it is done via the moment generating function method.

Next, we shall quote a theorem of Rogozin \cite{rogozin}.

\begin{theorem}[Rogozin, \cite{rogozin}]\label{rogozin}
For any random vector $v=(v_1,...,v_n)$ with independent coordinates having bounded concentration functions with parameters $a$ and $b$, and for any vector $u\in\R^n$, one has, for any $\epsilon>\max_{i=1,...,N}ca|u_i|$, that
$$\sup_{z\in\R}P(|\langle v,u\rangle-z|<\epsilon)\leq \frac{C\epsilon}{|u|}.$$
Here $c$ is an absolute constant and $C$ depends only on $a$ and $b.$
\end{theorem}

The next Lemma was shown by Rebrova and Tikhomirov \cite{RebTikh} as a consequence of Rogozin's Theorem \cite{rogozin}.

\begin{lemma}[consequence of Rogozin's theorem]\label{smallball}
Let $X$ be a random vector in $\R^n$ with independent coordinates whose distributions have concentration function separated from 1 (with constants $a$ and $b$). Then there exist constants $u\in\R$ and $v\in(0,1)$ (dependent only on $a$ and $b$) so that for an arbitrary $y\in\sfe,$ one has
$$\sup_{z\in\R}P(|\langle X, y\rangle-z|\leq u)\leq v.$$
\end{lemma}

Further, we shall make use of the following simple corollary of Rogozin's theorem.

\begin{lemma}[Another corollary of Rogozin's theorem]\label{rogozin-lemma}
Fix $C_1, C_2>0$. Consider and vector $u\in\mathbb{S}^{n-1}$ such that $\#\{i:\, |u_i|\geq \frac{c_1}{\sqrt{n}}\}\geq c_2n.$ For any random vector $v=(v_1,...,v_n)$ with independent coordinates having bounded concentration functions with parameters $a$ and $b$, one has, for any $\epsilon>\frac{C_1}{\sqrt{n}}$, that
$$\sup_{z\in\R}P(|\langle v,u\rangle-z|<\epsilon)\leq C_2\epsilon.$$
Here constants $C_1$ and $C_2$ depend only on $a$ and $b,$ and $c_1, c_2.$
\end{lemma}
\begin{proof} Let $\sigma=\{i:\, |u_i|\geq \frac{c_1}{\sqrt{n}}\}$. By our assumption, $\#\sigma\geq c_2n.$ Consider the random variable 
$$R:=\sum_{j\not\in\sigma} u_j v_j,$$
and note that
$$\langle u,v\rangle=R+\sum_{j\in\sigma} u_j v_j.$$
Observe that $R$ is independent of $\sum_{j\in\sigma} u_j v_j,$ and therefore
$$\sup_{z\in\R}P(|\langle v,u\rangle-z|<\epsilon)=P\left(|R+\sum_{j\in\sigma} u_j v_j-z|<\epsilon\right)\leq sup_{z\in \R} P(|\sum_{j\in\sigma} u_j v_j-z|<\epsilon).$$
Note that $\sum_{j\in\sigma} u_j^2\geq C''$. By Rogozin's theorem, for any $\epsilon>ca\cdot \frac{c_1}{\sqrt{n}},$ the right hand side of the above is bounded from above by $C(a,b)\epsilon$, finishing the proof.
\end{proof}

Lemma \ref{tensorization} and Lemma \ref{smallball} imply:

\begin{lemma}\label{smallballmatrix}
For any matrix $A$ with independent entries which have bounded concentration function, there exist absolute constants $c$ and $C$ (which depend only on $a$ and $b$) such that for any $x\in\sfe,$
$$P\left(||Ax||\leq c\sqrt{N}\right)\leq e^{-CN}.$$
\end{lemma}

Next, we shall need more elaborate estimates on the small ball probability; Rudelson and Vershynin \cite{RudVer-square}, \cite{RudVer-general}, \cite{RudVer-delocalization} have introduced for such purposes a notion of ``essential least common denominator''.

\begin{definition}[Rudelson-Vershynin]\label{LCD}
Fix $\alpha,c>0.$ For a vector $a\in\R^m,$ the least common denominator (LCD), is defined as
$$LCD_{\alpha,c}(a):=inf\{\theta>0:\,dist(\theta a,\mathbb{Z}^m)<\min(c|\theta a|,\alpha)\}.$$
Further, the LCD of a subspace $H$ in $\R^m$ is defined as
$$LCD_{\alpha,c}(H):=\inf_{v\in\mathbb{S}^{m-1}\cap H} LCD_{\alpha,c}(v).$$
Lastly, LCD of a matrix $A$ with rows $\{a_1,...,a_N\}\subset\R^m$ is defined as
$$LCD_{\alpha,c}(A):=inf\{|\theta|:\,\theta\in\R^m,\, dist(A\theta,\mathbb{Z}^N)<\min(c|A\theta|,\alpha)\}.$$
Note that when $m=1$, the matrix version of the LCD specializes to the LCD of a vector in $\R^N$.
\end{definition}

The next theorem is a deep result of Rudelson and Vershynin appearing as Theorem 3.3 in \cite{RudVer-general}.

\begin{theorem}[Rudelson-Vershynin]\label{small-ball}
Let $M:\R^N\rightarrow\R^m$ be a matrix such that for all $x\in\R^m,$ one has $|M^Tx|\geq |x|$. 
Let $X\in\R^N$ be a random vector with i.i.d. coordinates whose concentration functions are separated from $1.$ Then for every $\alpha>0,$ $c\in(0,1)$ and 
$$\epsilon>\frac{\sqrt{m}}{LCD_{\alpha,c}(M)},$$
one has
$$P\left(|MX|\leq \epsilon\sqrt{m}\right)\leq \left(\frac{C_1\epsilon}{c}\right)^m+C_2^m e^{-c\alpha^2}.$$
Here $C_1,$ $C_2$ and $c$ are absolute constants which only depend on the concentration function bound.
\end{theorem}

We shall need two corollaries of Theorem \ref{small-ball}, which are derived in \cite{RudVer-general} under the names Theorem 4.2 and Lemma 4.6 correspondingly.

\begin{corollary}[distance to a general subspace, Rudelson-Vershynin \cite{RudVer-general}]\label{Cor-random}
Let $m$ be an integer smaller than $N$. Let $X$ be a random vector in $\R^N$ with i.i.d. coordinates having concentration function separated from $1$ with constants $\tilde{a}$ and $\tilde{b}$, and let $H$ be a fixed $(N-m)$-dimensional subspace of $\R^N$. Pick any $v\in\R^N$. Then for every $\epsilon>\frac{\sqrt{N}}{LCD_{\alpha,c}(H^{\perp})}$, one has
$$P\left(dist(X,H+v)\leq \epsilon\sqrt{m}\right)\leq \left(\frac{C_1\epsilon}{c}\right)^m+C_2^m e^{-C_3\alpha^2},$$ 
where $C_1,$ $C_2$ and $C_3$ depend only on $\tilde{a}$ and $\tilde{b}$.
\end{corollary}

\begin{corollary}[Rudelson-Vershynin \cite{RudVer-general}]\label{Cor-determ}
Fix $D>0$. Let $x\in\R^N$ have $LCD_{\alpha,c}(x)\geq D.$ Fix a positive integer $m\leq \tilde{c}N,$ for an appropriate $\tilde{c}>0.$ Let $A$ be a random matrix with independent entries, i.i.d. rows and bounded concentration function of the entries, and let $B$ be an $(N-m)\times N$ submatrix of $A^T.$ Then for every $t>0$ we have
$$P(|Bx|<t\sqrt{N})\leq \left(Ct+\frac{C}{D}+Ce^{-c\alpha^2}\right)^{N-m},$$
where $c$ and $C$ depend only on the concentration function bounds.
\end{corollary}

We remark that the corresponding results in \cite{RudVer-general} are formulated with the sub-gaussian assumption, however it was only used to derive the bound on the concentration function, which we assume explicitly instead. We also emphasize that only rows (and not columns) of $A$ need to be i.i.d. for Corollary \ref{Cor-determ} to hold. Finally, we emphasize that the mean zero assumption is not required for the small ball estimates either, since the statement is translation invariant, as was later explained by Rudelson, Vershynin in \cite{RudVer-delocalization}. Also, Rudelson and Vershynin state an additional hypothesis that the LCD is bounded by $2D,$ however it is not required either, as can be easily verified.



\section{Tall and compressible cases.} 

We shall follow the scheme developed by Rudelson and Vershynin in \cite{RudVer-square}: that is, we shall consider a decomposition of the sphere to the set of ``compressible'' and ``incompressible'' vectors. Such decomposition, in fact, goes back to Litvak, Pajor, Rudelson, Tomczak-Jaegermann \cite{LPRT}, and was used in many papers by Rudelson, Vershynin, Tatarko, Tikhomirov, Rebrova et al \cite{RudVer-general}, \cite{tatarko}, \cite{Tikh}, \cite{RebTikh}.

Let $\rho,\delta>0$. As was discussed in the beginning of Section 2, a vector is $\delta$-sparse if at most $\delta n$ of its coordinates are non-zero. Compressible vectors $Comp(\delta,\rho)$ are such vectors on the sphere that are euclidean distance at most $\rho$ from the collection of $\delta-$sparse vectors. Incompressible vectors  $Incomp(\delta,\rho)$ are the vectors which are not compressible.

\medskip

\textbf{Notation.} 
\begin{itemize}
\item $Comp(\delta, \rho):=\{x\in\sfe:\, \exists y\in Sparse(\delta n)\,s.t.\, |x-y|\leq \rho\},$
\item $Incomp(\delta,\rho):=\sfe\setminus Comp(\delta,\rho).$
\end{itemize}

We shall begin with proving Proposition \ref{tallnodep}, and later switch to the beginning of the proof of Theorems \ref{main1} and \ref{main}.

\subsection{Proof of Proposition \ref{tallnodep}.} By the assumption of the proposition, $\E||A||_{HS}^2\leq KnN,$ and therefore,
$$P(||A||_{HS}^2\geq 10KnN)\leq 0.1.$$
Consequently,
\begin{equation}\label{condit}
P\left(\sigma_n(A)\leq C\sqrt{N}\right)\leq P\left(\sigma_n(A)\leq C\sqrt{N}\,|\,||A||_{HS}^2\leq 10KnN\right)+0.1.
\end{equation}
Let $\mathcal{N}$ be the net from Lemma \ref{HS}, with $\epsilon=1.$ Provided that the constant $C$ is chosen appropriately, the union bound together with Lemma \ref{HS} yields that
\begin{equation}\label{netHS}
P\left(\sigma_n(A)\leq C\sqrt{N}\,|\,||A||_{HS}^2\leq 10KnN\right)\leq e^{C'n}\cdot\sup_{x\in \frac{3}{2}B_2^n\setminus \frac{1}{2}B_2^n}P\left(|Ax|< C_1\sqrt{N}\right).
\end{equation}
It remains to note that Lemma \ref{tensorization} together with the assumptions of the Proposition imply that
\begin{equation}\label{smb-rows}
P(|Ax|< C_1\sqrt{N})\leq e^{-C_2N},
\end{equation}
for appropriate $C_1$ and $C_2.$ Suppose $C_0>0$ is such that for all $n\geq 1,$
$$e^{C'n}e^{-C_2C_0n}<0.1.$$
Letting $N\geq C_0n,$ we have, by (\ref{condit}), (\ref{netHS}), (\ref{smb-rows}) that, for an appropriate constants $C_0>0$ and $C>0,$ depending only on $a,b,K$, 
$$P(\sigma_n(A)<C\sqrt{N})<0.2.$$
An application of Markov's inequality finishes the proof of Proposition \ref{tallnodep}. $\square$

\subsection{Tall case: all entries independent.} We now proceed with the first step of the proof of Theorem \ref{main}.

\begin{proposition}\label{tall-final}
Fix $p>0$. Suppose $N$ and $n$ are integers. Suppose $A$ is an $N\times n$ random matrix with independent entries which have concentration function separated from 1. Suppose also that $A$ satisfies 
$$\sum_{i=1}^n\left(\mathbb{E}|Ae_i|^{2p}\right)^{\frac{1}{p}}\leq KnN e^{\frac{c_0 N}{n}},$$
with some absolute constant $K,$ and a sufficiently small constant $c_0$ (which depends on $K,$ $a$, $b$, $p$.) Suppose that $N\geq C'_0n$, for an appropriate $C'_0>0$ which depends only on $a, b, K,$ and $p$. 
Then, for appropriate absolute constants $C_1, C_2>0,$ which depend only on $p, a, b, K,$ and $C_1$ depends additionally on $p,$ one has 
$$P(\sigma_n\leq C_1\sqrt{N})\leq e^{-C_2\min(p,1)N}.$$
\end{proposition}
\begin{proof} By Lemma \ref{smallballmatrix}, there exist absolute constants  $C_0, C'_1>0$ such that
\begin{equation}\label{smallballrefer}
\sup_{x\in\sfe}P\left(|Ax|\leq C_1'\sqrt{N}\right)\leq e^{-C_0N}.
\end{equation}

Consider a net $\mathcal{N}$ from Theorem \ref{keytheoremnets} with parameters $\kappa=e^{\frac{C_0N}{10n}}$, $s=p^{-\frac{1}{p}}$, $\gamma=2,$ $\epsilon=e^{-\frac{c_0N}{n}}$, of size $\#\mathcal{N}\leq e^{C_3n+\left(\frac{C_0}{5}+c_0\right)N}$, for some absolute constant $C_3>1$. By (\ref{smallballrefer}), and the union bound, we have
\begin{equation}\label{netbnb-1}
P\left(\inf_{y\in\mathcal{N}} ||Ay||^2\leq C_1'N\right)\leq e^{C_3n+\left(\frac{C_0}{5}+c_0\right)N} e^{-C_0N}\leq e^{-c_1 N},
\end{equation}
provided that $N\geq C_0'n,$ for an appropriate constant $C_0'>0,$ and that $c_0$ is sufficiently small depending on $p, a, b, K.$

Next, by Theorem \ref{keytheoremnets}, with probability at least $1-e^{-cpN},$ we have
\begin{equation}\label{net-comp-2}
\sigma_n(A)\geq \inf_{y\in\mathcal{N}} ||Ay||-C''\sqrt{KN}.
\end{equation}
Combining (\ref{netbnb-1}) and (\ref{net-comp-2}), we have, for an appropriate choice of constants:
\begin{equation}\label{tall-concl}
P(\sigma_n\leq C_1\sqrt{N})\leq e^{-c_1N}+e^{-cpN},
\end{equation}
finishing the proof.
\end{proof}

\begin{remark} In view of Proposition \ref{tall-final}, it is enough to assume everywhere below that $N$ is not too large; in fact, we shall assume throughout the proof that $N\in [n,(c'+1)n]$. The case $N\in ((c'+1)n, C_0n]$ follows by considering sub-matrices; see the reasoning in the end of Section 8.
\end{remark}

In complete analogy with the Proposition \ref{tall-final}, we get

\begin{lemma}\label{comp-final}
Fix $p>0$. Suppose $A$ is an $N\times n$ random matrix with $N\geq n,$ all entries independent and having bounded concentration function with parameters $a$ and $b$, and such that $\sum_{i=1}^n(\E|Ae_i|^{2p})^{\frac{1}{p}}\leq KNn$. There exist absolute constants $\rho, \delta>0$ which depend (polynomially) only on $p$, $K,$ $a$ and $b,$ such that
$$P\left(\inf_{Comp(\delta,\rho)} |Ax|\leq C_1\sqrt{N}\right)\leq e^{-C_2\min(p,1) N},$$
with some constants $C_1$ and $C_2,$ which depend only on $K,a,b,$ and $C_1$ depends additionally on $p>0.$
\end{lemma}
\begin{proof} Let $\delta>0$ be sufficiently small, to be chosen later, and let $\rho=\delta^{c_0},$ for an appropriate $c_0$. Note that 
$$Comp(\delta, \gamma)=\cup_{H} (H+\rho B_{2}^n),$$
where the union runs over all the $\delta n-$dimensional coordinate subspaces, of which there is $\left(\frac{C'}{\delta}\right)^{\delta n}.$ Therefore, 
\begin{equation}\label{entropycomp}
N(Comp(\delta,\rho), \rho B_2^n)\leq \left(\frac{c}{\delta}\right)^{c'n\delta}=:e^{n\delta'}.
\end{equation}

Consider a net $\mathcal{N}\subset \frac{3}{2}B_2^n\setminus \frac{1}{2}B_2^n$ from Corollary \ref{cor1} with parameters $\mu=c\delta'$ and $\epsilon=\rho$, of cardinality $\#\mathcal{N}\leq e^{C_3\delta' N}$, for some absolute constant $C_3>1$.

In case $N\geq C_0 n$ the lemma follows from Proposition \ref{tall-final}. If $N\leq C_0n$, with probability at least $1-e^{-c(\delta)pn}=1-e^{-c'(\delta)pN}$, for every $x\in Comp(\delta,\rho)$ there exists $y\in\mathcal{N}$ such that
\begin{equation}\label{netcomp}
|Ay|\leq c|Ax|+C(\delta,p)\sqrt{N}.
\end{equation}
Therefore, by Lemma \ref{tensorization}, and the union bound, we have:
$$P\left(\inf_{Comp(\delta,\rho)} |Ax|\leq C_1\sqrt{N}\right)\leq e^{C_3\delta' N}e^{-C_2 N}+e^{-c'(\delta)p N}\leq e^{-c'\min(p,1)N},$$
provided that $\delta$ is selected small enough. 
\end{proof}


\section{Structure of random subspaces: an application of our net on the level sets of the LCD.}

Below we fix some notation and assumptions that shall be used in the present section.

\textbf{Notation}.
\begin{itemize}
\item Let $N\in [n, (c'+1)n]$, for an appropriate constant $c'>0$. Suppose $A$ is an $N\times n$ random matrix which has independent entries, i.i.d. rows, bounded concentration function. 

\item For an appropriate constant $\tilde{c}>0$, depending only on $K,$ $p$, $a$ and $b,$ consider an integer $1\leq m\leq \tilde{c}N$. Suppose, for each set $\sigma\subset \{1,...,N\}$ with $\#\sigma=N-m$, we have
\begin{equation}\label{suppose}
\sum_{i\in\sigma} \left(\E|A^Te_i|^{2p}\right)^{\frac{1}{p}}\leq Kn(N-m).
\end{equation}

\item Consider an $(N-m)-$dimensional random subspace
$$H=span\{Ae_i,\,i\in\sigma\}.$$

\item Fix the notation $B=(a_{ij})_{j\in\sigma}$ for an $(N-m)\times N$ submatrix of $A^T$. In other words, $H^{\perp}=Ker(B).$ 

\item Fix also the parameters $\delta,\rho>0$ from the previous section (which depend only on $p, a, b, K$), and consider the sets $Comp(\delta,\rho)$ and $Incomp(\delta,\rho)$.
\end{itemize}

\medskip

The main result of this subsection is the following Theorem, analogous in its statement and its proof to Theorem 4.1 from Rudelson-Vershynin \cite{RudVer-general}; the difference in the proof comes from an application of Theorem \ref{keytheoremnets}. This shall allow us to obtain the result in greater generality.

\begin{theorem}[distance to a random subspace]\label{smallballmain}
Let $X$ be a random vector in $\R^N$ with i.i.d. coordinates $X_i,$ and suppose that the concentration function of each $X_i$ is separated from 1 with constants $\tilde{a}$ and $\tilde{b}$. 
Suppose further that $X$ is independent of $Ae_i$ with $i\in\sigma$. Then for every $v\in\R^N$ and for every $\epsilon>0$ we have, with $H$ defined earlier,
$$P(dist(X,H+v)<\epsilon\sqrt{m})\leq (C\epsilon)^m+e^{-cN},$$
where the constants $C$ and $c$ depend only on $\tilde{a}, \tilde{b}, a, b, K, p$. 
\end{theorem}

In order to prove Theorem \ref{smallballmain}, we shall need the following

\begin{theorem}[structure]\label{structure}
$$P\left(LCD_{c_1\sqrt{N},c_2}(H^{\perp})\leq c_3\sqrt{N}e^{\frac{C_1N}{m}}\right)\leq e^{-C_2N}.$$
where $C_1, C_2, c_1, c_2, c_3$ are absolute constants, depending only on $a$ and $b$.
\end{theorem}

First, we quote the following Lemma by Rudelson and Vershynin \cite{RudVer-square}.

\begin{lemma}[Rudelson-Vershynin]\label{simplebnd}
Let $x\in Incomp(\delta,\rho)$ in $\R^N$. Then there exist $\tilde{C}, c_1, c_2$ which depend only on $\delta$ and $\rho$, such that
$$LCD_{c_1\sqrt{N},c_2}(x)\geq \tilde{C}\sqrt{N}.$$
\end{lemma}
The proof of this Lemma involves a ``restriction'' to the $\frac{\delta N}{2}$ coordinates and an application of the definition of the LCD.

Following the method of Rudelson and Vershynin, we shall improve this bound for random subspaces. Firstly, we note 

\begin{lemma}[random normal is incompressible]\label{rsai}
$$P\left(H^{\perp}\cap Comp(\delta,\rho)\neq\emptyset\right)\leq e^{-c\min(p,1)N},$$
if $c>0$ is sufficiently small, depending only on $a,b,K,p.$
\end{lemma}
\begin{proof} The event that there exists a vector $n\in H^{\perp}$ which is compressible implies that we have $Bn=0.$ According to Lemma \ref{comp-final}, applied to $B,$ this may only occur with probability not exceeding $e^{-c\min(p,1)N}$. Note that the application of Lemma \ref{comp-final} is justified by the assumption (\ref{suppose}).
\end{proof}

We follow the iteration argument of Rudelson-Vershynin \cite{RudVer-general}. Fix $\alpha=c_1\sqrt{n}$ and $c.$ Here $c_1$ and $c$ are small absolute constants which will be chosen, depending on our parameters $p, K, a, b$, later.

\begin{definition}
Fix $D\geq \tilde{C}\sqrt{N}$. Consider the level set of LCD
$$S_D:=\{x\in Incomp(\delta, \rho):\, D\leq LCD_{\alpha,c}(x)\leq 2D\}.$$
\end{definition}

\begin{lemma}\label{lcd-comp}
Fix $c\in (0,1)$ and suppose $\alpha\leq \frac{c\tilde{C}}{16}\sqrt{N},$ with $\tilde{C}$ from Lemma \ref{simplebnd}. For every $y\in S_D+\frac{\alpha}{2D\sqrt{N}}B_{\infty}^N$ we have $LCD_{\frac{\alpha}{2}, \frac{c}{4}}(y)\geq D.$
\end{lemma}
\begin{proof} By definition of the LCD, we have, for any $\tilde{D}<D$, and for any $x\in S_D,$ that 
\begin{equation}\label{lcd1}
dist(\tilde{D}x, \Z^N)\geq \min(\alpha, c\tilde{D}).
\end{equation}
Note that for any $y\in S_D+\frac{\alpha}{2D\sqrt{N}}B_{\infty}^N$ there exists an $x\in S_D$ such that
\begin{equation}\label{dist}
|\tilde{D}x-\tilde{D}y|\leq \frac{\alpha \tilde{D}}{2D}<\frac{\alpha}{2}.
\end{equation}
Combining (\ref{lcd1}) and (\ref{dist}), the fact that $|y|\leq 2$, and the fact that $\tilde{D}<D$ we conclude that
\begin{equation}\label{lcd2}
dist(\tilde{D}y, \Z^N)\geq c\tilde{D}-\frac{\alpha \tilde{D}}{2D}\geq \frac{c\tilde{D}}{2}\geq \frac{c\tilde{D}|y|}{4}, 
\end{equation}
where in the last inequality we used Lemma \ref{simplebnd}, together with our assumption on $\alpha,$ to conclude the proof of the Lemma.
\end{proof}

Next, we quote another result of Rudelson and Vershynin.

\begin{lemma}[Rudelson-Vershynin]\label{euclnet}
For an arbitrary $\alpha\in[0,\sqrt{N}]$, there exists a Euclidean $\frac{\alpha}{4D}$-net on $S_D$ of cardinality $\left(\frac{C_0 D}{\sqrt{N}}\right)^N.$ The constant $C_0$ is absolute (and does not depend on $\alpha$.)
\end{lemma}

\begin{remark} Formally, in \cite{RudVer-general} the authors consider a net of scale $\frac{4\alpha}{D}$ of cardinality $\left(\frac{C'_0 D}{\sqrt{N}}\right)^N,$ and this implies the existence of a $\frac{\alpha}{4D}$-net of cardinality $\left(\frac{C_0 D}{\sqrt{N}}\right)^N,$ with $C_0=48C_0'$.
\end{remark}

Lemmas \ref{lcd-comp} and \ref{euclnet}, together with Theorem \ref{keytheoremnets}, imply:

\begin{lemma}\label{goodnet}
There exist constants $c, c_0, C>0$ (which might depend on $p, a, b, K$), and an absolute constant $c_1>0$ (which does not depend on anything), such that for any $\alpha<c_0\sqrt{N}$, there exists a net $\mathcal{N}\subset \frac{3}{2}B_2^n\setminus \frac{1}{2}B_2^n$, with $\#\mathcal{N}\leq \left(\frac{C_1 D}{\sqrt{N}}\right)^N$ (where $C_1$ is an absolute constant not depending on anything) such that: 
\begin{enumerate}
\item with probability at least $1-e^{-C\min(p,1)N}$, for every $x\in S_D$ there is a $y\in \mathcal{N}$ such that 
$$|B(x-y)|\leq \frac{c_1\alpha \sqrt{N}}{D};$$
\item For every $y\in \mathcal{N}$, we have $LCD_{\frac{\alpha}{2}, \frac{c}{4}}(y)\geq D.$
\end{enumerate}
\end{lemma}
\begin{proof} By (\ref{suppose}), Theorem \ref{keytheoremnets} is applicable, with $s=\kappa=2$, $\gamma=1$ and $\epsilon=\frac{4\alpha}{D}$, and we get the first part, in view of Lemma \ref{euclnet}. The second part follows from Lemma \ref{lcd-comp}.
\end{proof}

\begin{remark} The level set of the LCD is the collection of points on the sphere which are close to certain scaled lattice. This fact is used by Rudelson and Vershynin in \cite{RudVer-square}, \cite{RudVer-general} to prove Lemma \ref{euclnet}; therefore, the euclidean net comes from the existence of the $\infty-$net, and we here, somewhat unnaturally, use the existence of the euclidean net to derive back the existence of the $\infty-$net. The paper is organized in this way merely for the sake of brevity.
\end{remark}

We arrive to the following

\begin{lemma}\label{infimum}
There exist $c_1, c_2, c_3, \mu\in (0,1)$, which may depend only on $p, a, b, K,$ such that the following holds. Let $\alpha=\mu\sqrt{N}$ and $D\leq c_1\sqrt{N}e^{\frac{c_1 N}{m}}$. Then
$$P\left(\inf_{x\in S_D}|Bx|\leq c_2\frac{N}{D}\right)\leq e^{-c_3\min(p,1)N}.$$
\end{lemma}
\begin{proof}
For the sake of brevity we merely sketch the proof of Lemma \ref{infimum}: it follows in exact the same manner as Lemma 4.8 from Lemmas 4.7 and Theorem 4.2 in Rudelson-Vershynin \cite{RudVer-general}. Namely, we note, by Lemma \ref{goodnet} and Corollary \ref{Cor-determ}, applied with $t=\frac{\zeta \sqrt{N}}{D}$ for a sufficiently small $\zeta\in (0,1),$ depending on $a$ and $b$:
$$P\left(\inf_{x\in S_D}|Bx|\leq c_2\frac{N}{D}\right)\leq e^{-c\min(p,1)N}+ \left(\frac{C'D}{\sqrt{N}}\right)^m \zeta^{N-m}\leq e^{-c\min(p,1)N}+e^{-cN},$$
where the last inequality holds for an appropriate choice of $\zeta,$ provided that $D\leq c_1\sqrt{N}e^{\frac{c_1 N}{m}}$, with appropriately chosen constants.
\end{proof}

\medskip

\textbf{Proof of Theorem \ref{structure}.} Following Rudelson-Vershynin \cite{RudVer-general}, we observe that
$$P\left(LCD_{c_1\sqrt{N}, c_2}(H^{\perp})\leq c_3\sqrt{N}e^{\frac{cN}{m}}\right)\leq$$
$$P(H^{\perp}\cap Comp(\delta,\rho)\neq\emptyset)+\sum_{D\in[\sqrt{N}, \sqrt{N}e^{\frac{cN}{m}}]} P(H^{\perp}\cap S_D\neq\emptyset).$$
It remains to note that the sum contains polynomially many summands and that $H^{\perp}\cap S_D\neq\emptyset$ implies that $Bx=0$; an application of Lemmas \ref{rsai} and \ref{infimum} finishes the proof. $\square$

\medskip

\textbf{Proof of Theorem \ref{smallballmain}.} The Theorem \ref{smallballmain} follows immediately from Theorem \ref{structure} and Corollary \ref{Cor-random}.


\section{Square matrix case.}

In this section we suppose that $N=n.$ We shall use the \textbf{notation}  
$$H_i:=span\{Ae_j,\,\,j\neq i\}.$$

We begin by citing the invertibility-via-distance Lemma from \cite{RudVer-square}. 
\begin{lemma}[Rudelson-Vershynin, Invertibility via distance]\label{invdist}
Fix a pair of parameters $\delta, \rho\in (0,\frac{1}{2})$. Then, for any $\epsilon>0,$
$$P\left(\inf_{x\in Incomp(\delta, \rho)} |Ax|\leq \epsilon \frac{\rho}{2\sqrt{\delta n}}\right)\leq \frac{1}{\delta n} \sum_{j=1}^n P(dist(Ae_j, H_j)\leq \epsilon).$$
\end{lemma}
The proof of the Lemma involves the pigeonhole principle and counting.

\medskip

\subsection{Proof of Theorem \ref{main1}.} \textbf{We begin with proving part 2.} Note that the assumptions of Theorem \ref{main1} imply (\ref{suppose}), and, together with the assumption of i.i.d. rows of $A,$ this allows us to apply Theorem \ref{smallballmain}, with $m=1$, and $X=Ae_j$, with $j=1,...,n$.

Let $n_j$ be a random normal, i.e. a vector normal to all columns of $A$ except for $Ae_j$. By Theorem \ref{smallballmain}, for any $\epsilon>e^{-c\min(p,1)n}$, with probability at least $1-e^{-c'\min(p,1)n}$, we get, for any $j,$
\begin{equation}\label{smb}
P(|\langle n_j, Ae_j\rangle|\leq \epsilon)\leq C\epsilon+e^{-C'\min(1,p)n}.
\end{equation}
We therefore have, by Lemma \ref{invdist} and (\ref{smb}), that
\begin{equation}\label{incomp}
P\left(\inf_{x\in Incomp(\delta,\rho)} |Ax|\leq\frac{\epsilon}{\sqrt{n}}\right)\leq C_1\epsilon+e^{-C_1'\min(1,p)n},
\end{equation}
with $C_1$ and $C_1'$ depending only on $\delta$ and $\rho,$ which in turn depend only on $K, a, b.$ Note that the estimate from (\ref{incomp}) is valid not only for $\epsilon>e^{-c\min(p,1)n}$, but for any $\epsilon>0$, since for $\epsilon\in[0,e^{-c\min(p,1)n}]$ we can estimate the said probability with $$P\left(\inf_{x\in Incomp(\delta,\rho)} |Ax|\leq\frac{e^{-c\min(p,1)n}}{\sqrt{n}}\right),$$ and apply the bound from (\ref{incomp}).

We conclude by (\ref{incomp}) and Lemma \ref{comp-final}, for $\epsilon\in[0, C_5]$, for an appropriate $C_5>0$:
$$P\left(\sigma_n(A)\leq \frac{\epsilon}{\sqrt{n}}\right)\leq P\left(\inf_{x\in Incomp(\delta,\rho)} |Ax|\leq\frac{\epsilon}{\sqrt{n}}\right)+P\left(\inf_{x\in Comp(\delta,\rho)} |Ax|\leq\frac{\epsilon}{\sqrt{n}}\right)\leq$$$$ e^{-c\min(1,p)n}+P\left(\inf_{x\in Comp(\delta,\rho)} |Ax|\leq C\sqrt{n}\right)$$$$\leq e^{-c\min(1,p)n}+e^{-c'\min(1,p)n}\leq e^{-c''\min(1,p)n}.$$

\textbf{Lastly, we outline part 1.} Suppose now that we do not assume that the rows of $A$ are i.i.d., and hence we may not invoke Theorem \ref{smallballmain}. Instead, we use Lemma \ref{rogozin-lemma}. Note that for each $j=1,...,n$, we have, in view of Lemma \ref{rsai} (which states that a random normal is incompressible):
\begin{equation}\label{smb-compar}
P(|\langle n_j, Ae_j\rangle|\leq \epsilon)\leq \sup_{u\in Incomp_{\delta,\rho}}P(|\langle u, Ae_j\rangle|\leq \epsilon).
\end{equation}
By Lemma \ref{rogozin-lemma}, there exists a constant $c=c(\rho,\delta)$, such that for every $\epsilon>\frac{c}{\sqrt{n}}$, the right hand side of (\ref{smb-compar}) is bounded from above by $C\epsilon$, with some constant $C$ depending only on $a$ and $b$ from the concentration function bound. The proof is therefore done. $\square$

\medskip

\begin{remark}\label{spiky} Actually, the proof implies that a more general result than Theorem \ref{main1} (part 1) holds: let $A$ be an $n\times n$ random matrix with independent uniformly anti-concentrated entries. Then there exist $\kappa\in (1,2]$, $s>0,$  and $C_1, C_2>0$, which depend only on the concentration function bound, such that for any $\epsilon>0,$
\begin{equation}\label{gen-concl}
P\left(\sigma_n(A)\leq \frac{\epsilon}{\sqrt{n}}\right)\leq C_1\epsilon+\frac{C_2}{\sqrt{n}}+P(\B_{\kappa}(A)\geq (1+s)n)+P(\B_{\kappa}(A^T)\geq (1+s)n).
\end{equation}

For example, consider a random matrix $B$ whose entries are all independent uniformly anti-concentrated bounded random variables (a weaker assumption which infers that all the entries are bounded simultaneously suffices), and let $A$ be a matrix with $a_{ij}=\beta_{ij}b_{ij}$, where $\beta_{ij}\geq 1$ are such constants that
$$\prod_{j=1}^n\left(\sum_{i=1} \beta^2_{ij}\right)\leq e^{cn},$$
and also
$$\prod_{i=1}^n\left(\sum_{j=1} \beta^2_{ij}\right)\leq e^{cn},$$
for a sufficiently small constant $c>0.$ Then part (1) of Theorem \ref{main1} is still valid for this matrix. Furthermore, the term $\frac{C_2}{\sqrt{n}}$ in (\ref{gen-concl}) can be replaced by $e^{-c_2 n}$, via refining an LCD-based argument (as shall be seen from a follow up paper joint with Tikhomirov and Vershynin).

This is not surprising that $\B_{\kappa}$ has better properties when the entries are bounded (or, say, sub-Gaussian): we have already seen that the deviation of $\B_{\kappa}$ improves when more moments are assumed, so when infinitely many moments are present, one may expect a stronger statement. Perhaps, a more general result, combining Theorem \ref{main1} with this observation, and interpolating over the cases of various moment assumptions, can be stated.
\end{remark}

\section{Arbitrary aspect ratio: proof of Theorem \ref{main} and more net applications.}

\textbf{Notation.} 
\begin{itemize}
\item In this section, we suppose that $N\in[n, (c'+1)n].$  Let $d=N+1-n\in [1,c'n]$ and note that $\sqrt{N+1}-\sqrt{n}=\frac{cd}{\sqrt{n}}$. 

\item For $J\subset \{1,...,n\}$ with $\#J=d$, fix the notation
$$H_{J^c}=span\{Ae_j: j\in J^c\},$$
and $\R^J=span\{e_i,\,i\in J\}.$ Note that $dim H_{J^c}=n-d$, and 
$$dim H^{\perp}_{J^c}=N-(n-d)=2d-1.$$

\item 
Denote
$$W:=P_{H^{\perp}_{J^c}}A|_{\R^J},$$
which is an $N\times d$ matrix with columns, independent conditionally on the realization of $H_{J_c}$.

\item We also denote
$$Spread_J=\left\{z\in \R^J:\, |z_k|\in [\frac{K_1}{\sqrt{d}}, \frac{K_2}{\sqrt{d}}]\right\},$$
for some appropriate absolute positive constants $K_1<K_2$, which may depend on our parameters $K, a, b$.

\item We shall assume that $A$ satisfies the assumptions of Theorem \ref{main}: all entries of $A$ are independent, mean zero, have bounded concentration function, the rows of $A$ are i.i.d., and for every $\sigma\subset\{1,...,n\}$ with $$\#\sigma=d,$$ we have
\begin{equation}\label{HSbound-chap8}
\E\sum_{i\in\sigma}|Ae_i|^2\leq KNd.
\end{equation}

\end{itemize}

\medskip
\medskip

We begin by quoting the \emph{invertibility via distance} Lemma from \cite{RudVer-general}.

\begin{lemma}[Rudelson, Vershynin, Lemma 6.2 in \cite{RudVer-general}]\label{invdistgeneral}
For any $\delta, \rho>0$ there exist constants $C_1, C_2>0$, and there exists a subset $J\subset[1...n]$ with $\#J=d$, such that for any $\epsilon>0,$
$$P\left(\inf_{x\in Incomp(\delta, \rho)}|Ax|\leq \frac{\epsilon d}{\sqrt{n}}\right)\leq C_1^d P\left(\inf_{z\in Spread_J} dist(Az, H_{J^c})\leq \epsilon\sqrt{d}\right).$$
\end{lemma}

\begin{remark} We note that in \cite{RudVer-general}, Rudelson and Vershynin state Lemma 6.2 \emph{for every $J$}, which follows from the i.i.d. assumption; without the i.i.d. assumption one only gets the existence, as stated here, however the exstance is the only thing which is required.
\end{remark}

We shall need the following

\begin{lemma}[projections of isotropic vectors]\label{isotr}
Let $X$ be a random vector in $\R^N$ such that $\E X=0$ and whose covariance matrix is a multiple of identity. Let $H$ be any $k-$dimensional subspace. Then
$$\frac{\E|P_H X|^2}{\E|X|^2}=\frac{k}{N}$$
\end{lemma}
\begin{proof} By scaling, without loss of generality one may assume that $X$ is isotropic, that is $Cov(X)=Id.$ Then $\E|X|^2=N$. It remains to verify that $\E|P_H X|^2=k.$ Indeed, there exists a collection of unit vectors $u_1,...,u_N\in H$, and a collection of constants $c_1,...,c_N$, such that
$$\sum_{i=1}^N c_i^2=k,$$
and moreover, for every $x\in\R^N,$
$$P_Hx=\sum_{i=1}^N c_i \langle u_i, x\rangle u_i,$$
and therefore,
$$|P_Hx|^2=\sum_{i=1}^N c^2_i \langle u_i, x\rangle^2.$$
Since $Cov(X)=Id$ and $\E X=0$,
$$\E|P_H X|^2=\E\sum_{i=1}^N c^2_i \langle u_i, X\rangle^2=\sum_{i=1}^N c_i^2=k.$$
The Lemma follows.
\end{proof}

As a corollary, we get

\begin{lemma}\label{Wcolumns}
For any $i\in J,$
$$
\mathbb{E}|We_i|^2=\frac{2d-1}{N}\mathbb{E}|Ae_i|^2.
$$
\end{lemma}
\begin{proof} Recall that $W:=P_{H^{\perp}_{J^c}}A|_{\R^J},$ and $dim H^{\perp}_{J^c}=2d-1$, while $Ae_i\in\R^N$. The Lemma follows from Lemma \ref{isotr} with $k=2d-1$ by integration, in view of the independence of $H_{J^c}$ and $Ae_i$ for $i\in J$.
\end{proof}

Our assumption (\ref{HSbound-chap8}), together with Lemma \ref{Wcolumns}, imply

\begin{corollary}\label{corW}
For an appropriate $C>0$ which depends only on $K$, for any $\kappa>1,$ 
$$P(\B_{\kappa}(W)\geq Cd^2)\leq (c\kappa)^{-2d},$$
where $c$ is an absolute constant depending only on $K.$
\end{corollary}
\begin{proof} By Lemma \ref{ldpB}, and in view of conditional independence of columns of $W$ when $H_{J^c}$ is fixed, we get
$$P(\B_{\kappa}(W)\geq 2\E||W||_{HS}^2\,|\,H_{J^c})\leq (c\kappa)^{-2d}.$$
By our assumption (\ref{HSbound-chap8}), together with Lemma \ref{Wcolumns}, we note that
$$\E||W||_{HS}^2=\E(||W||_{HS}^2\,|\, H_{J^c})\leq Kd(2d-1).$$
The Lemma follows by integration, in view of independence of $H_{J^c}$ and $Ae_i$ for $i\in J$.
\end{proof}

Further, we observe that an application of Theorem \ref{smallballmain} with $X=Az$ (in view of the fact that the rows of $A$ are i.i.d.), and $m=N-(n-d)=2d-1,$ yields, for any fixed vector $z$, any $t\geq e^{-\frac{cN}{d}}$, and any $w\in\R^n:$
\begin{equation}\label{smallballW}
P(|Wz-w|\leq t\sqrt{d})\leq (Ct)^{2d-1}.
\end{equation}

\medskip

As the first step, we derive

\begin{lemma}\label{asbefore}
For every $t\geq e^{-\frac{cN}{d}}$ and for any $w\in\R^n,$ we have
$$P\left(\inf_{z\in Spread_J}|Wz-w|\leq t\sqrt{d}, \, \B_{10}(W)\leq Cd^2\right)\leq (ct)^{d}.$$
\end{lemma}
\begin{proof} Recall that 
$$N(Spread_J, tB_2^n)\leq N(\mathbb{S}^{d-1}, tB_2^d)\leq\left(\frac{C}{t}\right)^{d-1}.$$
(For the last inequality see, e.g. Rudelson-Vershynin \cite{RudVer-general}.) By Theorem \ref{main-net}, applied with $\epsilon=t,$ $p=2,$ $n=d,$ $\gamma=2$, $s=1$ and $\kappa=10$, we conclude that there exists a net $\mathcal{N}\subset \frac{3}{2}B_2^d\setminus \frac{1}{2}B_2^d,$ such that
\begin{equation}\label{netbnd}
\#\mathcal{N}\leq \left(\frac{C'}{t}\right)^{d-1},
\end{equation} 
and
\begin{equation}\label{compar-sect8}
P\left(\inf_{z\in Spread_J}|Wz-w|\leq t\sqrt{d}, \, \B_{10}(W)\leq Cd^2\right)\leq P(\inf_{z\in \mathcal{N}}|Wz-w|\leq c't\sqrt{d}).
\end{equation}
Combining (\ref{smallballW}), (\ref{netbnd}) and (\ref{compar-sect8}), we estimate the probability in question with $(Ct)^{d}$, finishing the proof.
\end{proof}

\begin{remark}
Selecting an optimal $\kappa=\epsilon^{-\frac{1}{3}}$ in the argument above, in place of $\kappa=10,$ we could get the answer $(C\epsilon)^{\frac{2d-1}{3}}+e^{-CN}$ in Theorem \ref{main}. In what follows, we argue to improve it to the estimate involving only a logarithmic error.
\end{remark}

We mimic the decoupling argument from Rudelson-Vershynin \cite{RudVer-general}: we follow their idea of decoupling and iteration, replacing the norm of a sub-Gaussian matrix with the regularized HS-norm $\B_{\kappa}(W)$. Unlike in \cite{RudVer-general}, the iteration will be in the parameter $\kappa.$

\begin{lemma}[decoupling]\label{decoupling}
Let $d\geq 2$. Let $W$ be an $N\times d$ random matrix with independent columns. Let $z\in\frac{3}{2}B_2^d\setminus \frac{1}{2}B_2^d$ be such that $|z_k|\geq \frac{c_1}{\sqrt{d}}$, for some absolute constant $c_1$. Then for every $0<a<b$ we have
$$P(|Wz|<a, \B_{\kappa^2}(W)>b)\leq 2P\left(\B_{\kappa}(W)\geq \frac{b}{2}\right)\sup_{x\in\frac{3}{2}B_2^d\setminus \frac{1}{2}B_2^d, w\in\R^N} P(|Wx-w|<c_2 a).$$
\end{lemma}
\begin{proof} Decompose $[1,...,d]=I\cup H$ where $I$ and $H$ are disjoint and either equal in cardinality, or differ by 1. Consider the decompositions $W=W_I+W_H$ where $W_I$ has columns from $I$ and $W_H$ has columns from $H$. Note that
$$\B_{\kappa^2}(W)\leq \B_{\kappa}(W_I)+\B_{\kappa}(W_H).$$ 
For $z$ satisfying the assumption of the Lemma, consider $z_I=z|_I$ and $z_H=z|_H$, and note that both $z_I$ and $z_H$ satisfy $|z_I|\geq c'_1,$ and $|z_H|\geq c'_1.$ 
Observe that
$$P(|Wz|<a,\,\B_{\kappa^2}(W)>b)=p_I+p_H,$$
where
$$p_I=P\left(|Wz|<a, \B_{\kappa}(W_H)>\frac{b}{2}\right),$$
and 
$$p_H=P\left(|Wz|<a, \B_{\kappa}(W_I)>\frac{b}{2}\right).$$
We have $Wz=W_Iz_I+W_Hz_H$, and using the independence, we have
$$p_I\leq \sup_{w\in\R^N}P\left(|W_I z_I-w|<a\right)P\left(\B_{\kappa}(W_H)>\frac{b}{2}\right)\leq$$
$$\sup_{w\in\R^N}P(|Wz_I-w|<a)P\left(\B_{\kappa}(W_H)>\frac{b}{2}\right),$$
since $Wz_I=W_Iz_I.$ Same estimate we do for $W_H.$ To conclude the proof it remains to recall that $|z_I|\geq c,$ for some absolute constant $c>0.$
\end{proof}

Next, we observe

\begin{lemma}\label{policemen}
Let $W$ be as before, and let $t\in [e^{-\frac{cN}{d}},\frac{c_1}{10}]$. Then, for any $\kappa\geq 10,$
$$P(\inf_{z\in Spread_J} |Wz|\leq t\sqrt{d}, \,\,\B_{\kappa}(W)\geq Cd^2,\,\, \B_{2\kappa}(W)\leq Cd^2)\leq e^{-cN}+ (Ct)^d (\log\kappa)^{d-1}.$$
\end{lemma}
\begin{proof} For the set $S=Spread_J$, consider the net $\mathcal{N}\subset \frac{3}{2}B_2^d\setminus \frac{1}{2}B_2^d$ of size $\left(\frac{C'}{t}\right)^{d-1}$, given by Theorem \ref{main-net}, applied with $\gamma=2,$ $\epsilon=t,$ and $2\kappa$. By the union bound,
$$P\left(\inf_{z\in Spread_J} |Wz|\leq t\sqrt{d}, \,\,\B_{\kappa}(W)\geq d^2,\,\, \B_{2\kappa}(W)\leq d^2\,|\, H_{J^c}\right)\leq$$ 
\begin{equation}\label{eq1-L-S8}
(Ct)^{-(d-1)}\kappa^d(\log\kappa)^{d-1} \sup_{z\in Spread_J+\frac{2t}{\sqrt{d}}B_{\infty}^d,w\in\R^d}P\left(|Wz-w|\leq Ct\sqrt{d}, \,\,\B_{\kappa}(W)\geq d^2\,|\, H_{J^c}\right).
\end{equation}
Note that $z\in Spread_J+\frac{2t}{\sqrt{d}}B_{\infty}^d$ satisfies the assumption of the Lemma \ref{decoupling}; further, observe that conditioning on $H_{J^c}$ ``makes'' the columns of $W$ independent, hence allowing to apply Lemma \ref{decoupling}. By Lemma \ref{decoupling}, (\ref{eq1-L-S8}) is estimated by
\begin{equation}\label{eq2-L-S8}
(Ct)^{-(d-1)}\kappa^d(\log\kappa)^{d-1} \sup_{z,w}P(|Wz-w|\leq Ct\sqrt{d}\,|\, H_{J^c})P( \B_{\sqrt{\kappa}}(W)\geq C'd^2\,|\, H_{J^c}),
\end{equation}
and by (\ref{smallballW}), together with Lemma \ref{ldpB}, we estimate (\ref{eq2-L-S8}) by
$$(Ct)^{-(d-1)}\kappa^d(\log\kappa)^{d-1} t^{2d-1}(c'\sqrt{\kappa})^{-2d}\leq (Ct)^d (\log\kappa)^{d-1}.$$
The Lemma follows by integration, in view of independence of $H_{J^c}$ and $Ae_i$ for $i\in J$.

\end{proof}

Finally, we arrive to

\begin{proposition}\label{keyfinal}
For every $\epsilon\geq e^{-\frac{cN}{d}}$, we have
$$P\left(\inf_{z\in Spread_J}|Wz|\leq \epsilon\sqrt{d}\right)\leq \left(c\epsilon\log\frac{1}{\epsilon}\right)^d.$$
\end{proposition}
\begin{proof} Fix $s\geq 10$. Note that $\B_{10}(W)\geq \B_{20}(W)\geq...\geq \B_{2^{s}\cdot 10}$. Therefore, either $Cd^2\geq \B_{10}(W)$, or $Cd^2\in [B_{2^{k+1}\cdot10}, \B_{2^{k}\cdot10}]$, for some $k=1,...,s$, or $Cd^2\leq \B_{2^{s}\cdot 10}.$

Hence, we have
$$P\left(\inf_{z\in Spread_J} dist(Az, H_{J^c})\leq \epsilon\sqrt{d}\right)\leq$$
$$P\left(\inf_{z\in Spread_J} dist(Az, H_{J^c})\leq \epsilon\sqrt{d}, \B_{10}(W)\leq Cd^2\right)+$$
$$\sum_{\kappa=10, 20,...,\frac{1}{\epsilon}} P\left(\inf_{z\in Spread_J} dist(Az, H_{J^c})\leq \epsilon\sqrt{d}, \,\,\B_{\kappa}(W)\geq Cd^2,\,\, \B_{2\kappa}(W)\leq Cd^2\right)+$$
$$P(\B_{\frac{1}{\epsilon}}(W)\geq Cd^2).$$
The first term is estimated with $(c\epsilon)^{d}$ by Lemma \ref{asbefore}. The last term is estimated with $(C\epsilon)^{2d}$ by Lemma \ref{ldpB}. The middle sum is estimated by Lemma \ref{policemen} with
$$\sum_{\kappa=10, 20,...,\frac{1}{\epsilon}} (C\epsilon)^d(\log\kappa)^{d-1}\leq \left(C\epsilon\log\frac{1}{\epsilon}\right)^{d},$$
and the proof is done.
\end{proof}

\medskip
\medskip

\textbf{Proof of Theorem \ref{main}.} The case $N=n$ was treated in the previous section. In the case $N\in [n+1,(c'+1)n]$, we combine Lemma \ref{comp-final}, Lemma \ref{invdistgeneral} and Proposition \ref{keyfinal}, to get the estimate 
\begin{equation}\label{eq1-final}
P\left(\sigma_n(A)\leq \frac{\epsilon d}{\sqrt{n}}\right)\leq \left(C\epsilon\log\frac{1}{\epsilon}\right)^d+e^{-cN},
\end{equation}
for $\epsilon\geq e^{-c\frac{N}{d}}$; for $\epsilon\in (0, e^{-c\frac{N}{d}})$ the bound follows then automatically. Note that the assumptions of Theorem \ref{main} allow to apply Lemma \ref{comp-final}, Lemma \ref{invdistgeneral} and Proposition \ref{keyfinal}.

In the case $N\in [(c'+1)n, C_0n]$, we apply (\ref{eq1-final}) with $\tilde{A}$, an $(c'+1)n\times n$ submatrix of $A$. Note that $\tilde{A}$ still satisfies the assumptions of Theorem \ref{main}, since the rows of $A$ are i.i.d. The proper constant adjustment settles this case. 
Finally, for $N\geq C_0 N$ the bound follows from Proposition \ref{tall-final}. $\square$

\begin{remark} We remark that the condition of the rows being i.i.d comes from the fact that Theorem \ref{small-ball} requires an i.i.d. assumption. Corollary \ref{Cor-determ} hence requires i.i.d. rows for $A$, and an application of Corollary \ref{Cor-random} with $X=Az$ further requires i.i.d. rows.

In addition to that, in the tall case (Theorem \ref{main}), the assumption of the rows having the same second moments is needed in Lemma \ref{Wcolumns}, and in a few other places.

The mean zero assumptions on the entries, made in Theorem \ref{main}, comes up in the tall case in Lemma \ref{Wcolumns}.

The assumption of independence of columns of $A$ (and sometimes, also of $A^T$) is required throughout, when applying Lemmas \ref{tensorization} and \ref{smallball}, as well as Theorem \ref{keytheoremnets}. 

Bounded concentration function is key for all the small-ball results: Lemmas \ref{tensorization} and \ref{smallball} as well as Theorem \ref{small-ball}.

Finally, the assumptions (\ref{HSbound-11}) and (\ref{HSbound-12}) arise from applications of Theorem \ref{keytheoremnets} to $A$ (when estimating the compressible case), and $(n-1)\times n$ submatricies of $A^T$ (when constructing nets on the level sets of LCD). 

Similarly, the assumption (\ref{HSbound}) is needed due to applications of Theorem \ref{keytheoremnets} to $N\times d$ restrictions of $A$, and it further permits an application of Theorem \ref{keytheoremnets} to $A$ when studying the compressible case. Together with the i.i.d. rows assumption, (\ref{HSbound}) is also needed to allow an application of Theorem \ref{keytheoremnets} to $(N-2d+1)\times N$ submatrices of $A^T$ when constructing nets on the level sets of LCD.
\end{remark}


\begin{thebibliography}{9}



\bibitem{KA} N. Alon, B. Klartag, \emph{Optimal compression of approximate inner products and dimension reduction}, Symposium on Foundations of Computer Science (FOCS 2017), 639-650. 

\bibitem{BaiYin} Z. D. Bai and Y. Q. Yin, \emph{Necessary and sufficient conditions for almost sure convergence of the largest eigenvalue of a Wigner matrix,} Ann. Probab. 16 (1988), no. 4, 1729-1741. MR0958213

\bibitem{BaiYin-small} Z. D. Bai,Y. Q. Yin, \emph{Limit of the smallest eigenvalue of a large-dimensional sample covariance matrix,} Ann. Probab. 21 (1993), 1275-1294.

\bibitem{beck} J. Beck, \emph{Irregularities of distribution,} I. Acta Math. 159 (1987), no. 1-2, 1-49.












\bibitem{BVW} J. Bourgain, V. H. Vu and P. M. Wood, \emph{On the singularity probability of discrete random matrices,} J. Funct. Anal. 258 (2010), no. 2, 559-603. MR2557947

\bibitem{cook} N. Cook, \emph{Lower bounds for the smallest singular value of structured random matrices,} Ann. Probab. Volume 46, Number 6 (2018), 3442-3500.

\bibitem{edelman} A. Edelman, \emph{Eigenvalues and condition numbers of random matrices}, SIAM J. Matrix Anal. Appl. 9 (1988), 543-560.

\bibitem{FS} O. N. Feldheim and S. Sodin, \emph{A universality result for the smallest eigenvalues of certain sample covariance matrices,} Geom. Funct. Anal. 20 (2010), no. 1, 88-123. MR2647136

\bibitem{feller} W. Feller, \emph{A limit theorem for random variables with infinite moments,} Amer. Journal of Math, (1946), 68, 257-262.


\bibitem{gordon} Y. Gordon, \emph{Some inequalities for Gaussian processes and applications,} Israel J. Math. 50 (1985), 265-289.

\bibitem{GLT} O. Guedon, A. Litvak, K. Tatarko, \emph{Random polytopes obtained by matrices with heavy tailed entries}, Commun. in Cont. Mathematics, (2020).

\bibitem{vHLY} R. Van Handel, R. Lata\l{}a, P. Youssef, \emph{The dimension-free structure of nonhomogeneous random matrices}, Invent. Math. 214, 1031-1080 (2018). 











\bibitem{KLkold} B. Klartag, G. V. Livshyts, \emph{The lower bound for Koldobsky?s slicing inequality via random rounding}, to appear in Lect. notes GAFA seminar.







\bibitem{LitSpec} A. Litvak, S. Spektor, \emph{Quantitative version of a Silverstein's result,} GAFA, Lecture Notes in Math., 2116 (2014), 335-340.

\bibitem{LugMen} G. Lugosi, and S. Mendelson, \emph{Risk minimization by median-of-means tournaments,} Journal of the European Mathematical Society, to appear, 2017.

\bibitem{KV} R. Kannan, S. Vempala, \emph{Sampling Lattice Points}, Proc. 29th ACM Symposium on the Theory of Computing (STOC '97), El Paso, (1997), Invited for publication in Journal of Comp. and System Sciences.

\bibitem{KKS} J. Kahn, J. Komlos, E. Szemeredi, \emph{On the probability that a random $\pm 1$ matrix is singular,} J. Amer. Math. Soc. 8 (1995), 223-240.

\bibitem{MenKol} V. Koltchinskii, S. Mendelson, \emph{Bounding the smallest singular value of a random matrix without concentration}, International Mathematics Research Notices, Vol. 2015 (23), 12991-13008, 2015.




\bibitem{Latala} R. Lata\l{}a, \emph{Some estimates of norms of random matrices}, Proc. Amer. Math. Soc. 133 (2005), 1273-1282.

\bibitem{LitOf} J. E. Littlewood, A. C. Offord, \emph{On the number of real roots of a random algebraic equation. III,} Rec. Math. [Mat. Sbornik] N.S. 12 (54), (1943), 277-286

\bibitem{LPRT} A. Litvak, A. Pajor, M. Rudelson, N. Tomczak-Jaegermann, \emph{smallest singular value of random matrices and geometry of random polytopes,} J. Reine Angew. Math. 589 (2005), 1-19.



\bibitem{LytTikh} A. Lytova, K. Tikhomirov, \emph{On delocalization of eigenvectors of random non-Hermitian matrices}, preprint.

\bibitem{Neuman} J. von Neumann, H. H. Goldstine, \emph{Numerical inverting of matrices of high order}, Bull. Amer. Math. Soc. 53 (1947), 1021-1099.

\bibitem{MenPao} S. Mendelson, G. Paouris, \emph{On the singular values of random matrices,} Journal of the European Mathematics Society, 16, 823-834, 2014.


\bibitem{RebTikh} E. Rebrova, K. Tikhomirov, \emph{Coverings of random ellipsoids, and invertibility of matrices with i.i.d. heavy-tailed entries}, Israel Journal of Math, to appear.




\bibitem{rogozin} B. A. Rogozin, \emph{An estimate for the maximum of the convolution of bounded densities}, Teor. Veroyatnost. i Primenen. 32 (1987), no. 1, 53-61, English translation:Theory Probab. Appl. 32 (1987), no. 1, 48-56.

\bibitem{Rud-square} M. Rudelson, \emph{Invertibility of random matrices: norm of the inverse}, Annals of Mathematics 168 (2008), 575-600.

\bibitem{Rud-notes} M. Rudelson, \emph{Delocalization of eigenvectors of random matrices Lecture notes,} preprint.

\bibitem{RudVer-square} M. Rudelson, R. Vershynin, \emph{The Littlewood-Offord problem and invertibility of random matrices}, Adv. Math.  218  (2008),  no. 2, 600-633.

\bibitem{RudVer-general} M. Rudelson, R. Vershynin, \emph{Smallest singular value of a random rectangular matrix,} Communications on Pure and Applied Mathematics 62 (2009), 1707-1739. 

\bibitem{RudVer-icm} M. Rudelson, R. Vershynin, \emph{Non-asymptotic theory of random matrices: extreme singular values,} Proceedings of the International Congress of Mathematicians, 2010, pp. 83-120.

\bibitem{RudVer-smallball} M. Rudelson, R. Vershynin, \emph{Small ball probabilities for linear images of high dimensional distributions,} Int. Math. Res. Not. 19 (2015), 9594-9617.

\bibitem{RudVer-delocalization} M. Rudelson, R. Vershynin, \emph{Delocalization of eigenvectors of random matrices with independent entries,} Duke Math. J.
Volume 164, Number 13 (2015), 2507-2538.

\bibitem{szarek} S. Szarek, \emph{Condition numbers of random matrices, J. Complexity 7 (1991),} 131-149.

\bibitem{Silverstein} J. Silverstein, \emph{On the weak limit of the largest eigenvalue of a large dimensional sample covariance matrix,} J. of Multivariate Anal., 30 (1989), 2, 307-311.

\bibitem{Srin} A. Srinivasan, \emph{Approximation Algorithms via Randomized Rounding: a Survey}, Lectures on Approximation and Randomized Algorithms, Series in Advanced Topics in Mathematics, Polish Scientific Publishers PWN, Warsaw, 9-71, (1999).

\bibitem{tatarko} K. Tatarko, \emph{An upper bound on the smallest singular value of a square random matrix}, preprint.

\bibitem{TaoVu-LitOf} T. Tao, V. Vu, \emph{Inverse Littlewood-Offord theorems and the condition number of random discrete matrices,} Annals of Math. 169 (2009), 595-632.

\bibitem{taovu-1} T. Tao, V. Vu, \emph{On random $\pm 1$ matrices: Singularity and Determinant,} Random Structures and Algorithms 28 (2006), no 1, 1-23.

\bibitem{taovu} T. Tao, V. Vu, \emph{On the singularity probability of random Bernoulli matrices}, J. Amer. Math. Soc. 20 (2007), 603-628.

\bibitem{taovu-tall} T. Tao and V. Vu, \emph{Random matrices: the distribution of the smallest singular values,} Geom. Funct. Anal. 20 (2010), no. 1, 260-297. MR2647142

\bibitem{Tikh} K. Tikhomirov, \emph{The limit of the smallest singular value of random matrices with i.i.d. entries,} Adv. Math. 284 (2015), 1-20.

\bibitem{Tikh-nomoments} K. Tikhomirov, \emph{The smallest singular value of random rectangular matrices with no moment assumptions on entries}, Israel J. Math. 212 (2016), no. 1, 289-314.

\bibitem{Tikh1} K. Tikhomirov \emph{Invertibility via distance for non-centered random matrices with continuous distributions}, preprint.

\bibitem{TikhErd} K. Tikhomirov \emph{Singularity of random Bernoulli matrices}, to appear in Annals of Math.


\bibitem{Versh} R. Vershynin, \emph{High-dimensional probability: an introduction with applications in data science}, Cambridge University Press, 2018.

\bibitem{Versh-tall} R. Vershynin, \emph{Spectral norm of products of random and deterministic matrices,} Probability Theory and Related Fields 150 (2011), 471-509.

\end{thebibliography}
\end{document}